\documentclass[a4paper]{article}

\usepackage{amsfonts,amsmath,amssymb,amsthm,epsfig,psfrag}
\usepackage[all]{xy}

\newtheorem{theorem}{Theorem}
\newtheorem{definition}{Definition}
\newtheorem{lemma}{Lemma}
\newtheorem{proposition}{Proposition}

\newtheorem{remark}{Remark}

\def\hpi{\widehat{\pi}}

\def\N{{\mathbb N}}
\def\EE{{\mathbb E}}
\def\PP{{\mathbb P}}
\def\RR{{\mathbb R}}
\def\ind{{\mathbf{1}}}
\def\ore{\overrightarrow{e}}
\def\orE{\overrightarrow{E}}

\def\d{\partial}
\def\bx{\bold{x}}
\def\bw{\bold{w}}
\def\bB{\bold{B}}

\def\Poi{\text{Poi}}
\def\Bin{\text{Bin}}
\def\odeg{\overline{\text{deg}}}
\def\Pcal{{\mathcal P}}
\def\Rcal{{\mathcal R}}
\def\Dcal{{\mathcal D}}
\def\Qcal{{\mathcal Q}}

\def\Gcal{{\mathcal G}}
\def\Fcal{{\mathcal F}}
\def\tp{\tilde{p}}

\def\Ucal{{\mathcal U}}
\def\BUT{{\mathcal{BUGWT}}}

\def\Yb{\bold{Y}}
\def\bY{\bold{Y}}
\def\bI{\bold{I}}
\def\bZ{\bold{Z}}
\def\bW{\bold{W}}
\def\bM{\bold{M}}
\def\Xb{\bold{X}}
\def\bX{\bold{X}}
\def\omu{\overline{\mu}}
\def\oq{\overline{q}}
\def\bac{\backslash}

\newcommand{\BEAS}{\begin{eqnarray*}}
\newcommand{\EEAS}{\end{eqnarray*}}
\newcommand{\BEA}{\begin{eqnarray}}
\newcommand{\EEA}{\end{eqnarray}}
\newcommand{\BEQ}{\begin{equation}}
\newcommand{\EEQ}{\end{equation}}
\newcommand{\BIT}{\begin{itemize}}
\newcommand{\EIT}{\end{itemize}}
\newcommand{\BNUM}{\begin{enumerate}}
\newcommand{\ENUM}{\end{enumerate}}

\begin{document}
\title{A new approach to the orientation\\ of random hypergraphs\footnote{A preliminary version of this paper appeared in \cite{soda12}.}}
\author{M. Lelarge\footnote{INRIA-ENS, Paris, France,
email: marc.lelarge@ens.fr}}
\date{}
\maketitle

\begin{abstract}
A $h$-uniform hypergraph $H=(V,E)$ is called $(\ell,k)$-orientable if
there exists an assignment of each hyperedge $e\in E$ to exactly
$\ell$ of its vertices $v\in e$ such that no vertex is assigned more
than $k$ hyperedges. Let $H_{n,m,h}$ be a hypergraph, drawn uniformly
at random from the set of all $h$-uniform hypergraphs with $n$
vertices and $m$ edges. In this paper, we determine the threshold of
the existence of a $(\ell,k)$-orientation of $H_{n,m,h}$ for $k\geq 1$
and $h>\ell\geq 1$, extending recent results motivated by applications
such as cuckoo hashing or load balancing with guaranteed maximum load.
Our proof combines the local weak convergence of sparse graphs and a
careful analysis of a Gibbs measure on spanning subgraphs with degree
constraints. It allows us to deal with a much broader class than the
uniform hypergraphs. 
\end{abstract}


{\bf AMS Subject Headings:} 68Q87, 68P05, 60C99.

\section{Introduction}

Motivated by load balancing problems \cite{sek03}, Gao and Wormald \cite{gawo10} introduced the following generalisation to random hypergraphs of a commonly studied orientation problem on graphs. A $h$-uniform hypergraph is a hypergraph such that all its hyperedges have size $h$. Let $h>\ell$ be two given positive integers. A hyperedge is said to be $\ell$-oriented if exactly $\ell$ distinct vertices in it are marked with positive signs with respect to the hyperedge. The indegree of a vertex is the number of positive signs it receives. Let $k$ be a positive integer. A $(\ell,k)$-orientation of an $h$-uniform hypergraph is a $\ell$-orientation of all hyperedges such that each vertex has indegree at most $k$. If such a $(\ell,k)$-orientation exists, we say that the hypergraph is $(\ell,k)$-orientable.
We consider $\Gcal_{n,m,h}$ the probability space of the set of all $h$-uniform hypergraphs on $n$ vertices and $m$ hyperedges with the uniform distribution. A random $h$-uniform hypergraph is then denoted by $H_{n,m,h}$. We are now ready to state our main result in this framework:
\begin{theorem}\label{the:hyperg}
Let $Q(x,y) = e^{-x}\sum_{j\geq y}\frac{x^j}{j!}$ and $\Bin(n,p)$ denote a binomial random variable with parameters $n\in \N$ and $p\in[0,1]$, i.e. $\PP(\Bin(n,p)=k) = {n\choose k}p^k(1-p)^{n-k}$. For integers $h>\ell\geq 1$, $k\geq 1$ with $\max(h-\ell,k)\geq 2$, 
let $\xi^*$ be the unique positive solution to
\BEAS
hk =\xi^*\frac{\EE\left[\max\left( \ell-\Bin(h,1-Q(\xi^*,k)),0\right) \right]}{Q(\xi^*,k+1)\PP\left( \Bin(h-1,1-Q(\xi^*,k))<\ell\right)}.
\EEAS
Let
\BEAS
c^*_{h,\ell,k} = \frac{\xi^*}{h\PP\left( \Bin(h-1,1-Q(\xi^*,k))<\ell\right)}.
\EEAS
Then
\BEAS
\lim_{n\to \infty}\PP\left( H_{n,\lfloor cn \rfloor, h} \mbox{ is $(\ell,k)$-orientable}\right) =
\left\{
\begin{array}{ll}
0&\mbox{ if } c>c^*_{h,\ell,k},\\
1&\mbox{ if } c<c^*_{h,\ell,k}.
\end{array}
\right.
\EEAS
\end{theorem}

The characterisation of the threshold $c^*_{h,k,\ell}$ in
\cite{gawo10} (for $k$ sufficiently large) involves the solution of a
differential equation system which is rather complicated (according to
the authors themselves) and does not allow to get explicit values for
$c^*_{h,k,\ell}$. We believe that our method of proof and the
characterisation of the threshold in Theorem \ref{the:hyperg} is much
simpler. 
Note that for the case $k=1$ and $\ell=h-1$, the threshold for
orientablity is equal to the threshold for the apparition of a giant
component so that we have
\BEAS
c^*_{h,h-1,1} =\frac{1}{h(h-1)}.
\EEAS
Hence our result allows to compute the orientation threshold on the
whole range of parameters $k\geq 1$ and $h>\ell\geq 1$.
To illustrate our result, we computed numerical values of the critical load $\frac{\ell c^*_{h,\ell,k}}{k}$ for different values of the parameters.
First when $\ell$ and $k$ vary while $h=4$ is fixed:

\begin{tabular}{c|cccccc}
$\ell\slash k$&1&2&3&4&5&6\\
\hline
1&0.9767701648&0.9982414840&0.9997951433&0.9999720662&0.9999958680&0.9999993570\\
2&0.7596968140&0.9266442602&0.9676950000&0.9834603210&0.9908051880&0.9946173050\\
3&0.0833333333&0.6612827547&0.7892143791&0.8525202000&0.8898186996&0.9141344769
\end{tabular}

Then when $h$ and $k$ vary while $\ell=2$ is fixed:

\begin{tabular}{c|cccccc}
$h\slash k$&1&2&3&4&5&6\\
\hline
5&0.8833250296&0.9730747564&0.9909792334&0.9964896324&0.9985201920&0.9993444714\\
6&0.9378552354&0.9894605852&0.9974188480&0.9992698236&0.9997769140&0.9999284650\\
7&0.9652101902&0.9957801256&0.9992689074&0.9998543770&0.9999687056&0.9999929390
\end{tabular}

We should also stress that our result (and our proof) unifies various results available for different ranges of the parameters.

The case $\ell=1$ has attracted a lot of attention and can be
described in different terminologies. In the balanced allocation
paradigm \cite{abku99}, we have $m$ balls and $n$ bins. To each ball,
two bins are assigned at random. Each ball is to be placed in one of
the two bins assigned to it; the aim is to keep the maximal load
small. Another formulation of the same problem can be given in data
structure language: in the cuckoo hashing method \cite{pr01}, each one
of $m$ keys is assigned two locations in a hash table of size $n$ and
can be stored in one of the two locations. If each location has
capacity one (or in the balanced allocation, maximal load cannot
exceed one), the offline version of this problem corresponds exactly
to the standard $1$-orientability of the classical random graph
$G_{n,m}$ drawn uniformly from the set of all graphs with $n$ vertices
and $m$ edges. To see the connection, associate to each location a
vertex of the graph and to each key an edge of the graph: the
orientation of the edges correspond to the allocation. It is easily
seen in this case that we must have $m<n/2$ in order for $G_{n,m}$ to
be $1$-orientable. An interesting generalisation of this problem
considers bins/locations of capacity $k\geq 1$ \cite{diewei07}. This
generalisation (with the number of choices per ball/key still equals
to $2$) corresponds to the $(1,k)$-orientability described above for
the random graph $G_{n,m}=H_{n,m,2}$. For $k\geq 2$, the sharp
threshold for the $k$-orientability of the random graph $G_{n,m}$,
corresponding to $c^*_{2,1,k}$ in our Theorem \ref{the:hyperg} was
found simultaneously by Cain, Sanders and Wormald \cite{csw07} and
Fernholz and Ramachandran \cite{fr07}. Another generalisation
allows for $h> 2$ choices of bins/locations. In this
case, the graph associated to the problem is a random
$h$-hypergraph. The $(1,1)$-orientability threshold of $h$-uniform
random hypergraph with $h\geq 3$, corresponding to our $c^*_{h,1,1}$
in Theorem \ref{the:hyperg}, has been independently computed by
Dietzfelbinger et al. \cite{dgmmpr10}, Fountoulakis and Panagiotou
\cite{fopa10} and Frieze and Melsted \cite{frpa09}. This corresponds
to a case where bins/locations have unit capacity. Recently, the
extension to capacity $k\geq 1$ has been solved by Fountoulakis, Kosha
and Panagiotou \cite{fkp11} for any $h\geq 3$ (refining results of
\cite{gawo10} for this particular case). Our derivation of $c^*_{h,1,k}$ agrees with \cite{fkp11}. In the generalisation proposed in \cite{gawo10}, each
batch of $\ell$ balls has $h>\ell$ choices of bins and each of the $n$ bins has a
capacity $k$. Our Theorem \ref{the:hyperg} gives the threshold
$c_{h,\ell,k}$ such that, as $n$ goes to infinity: if $n/m\leq c$ with
$c<c_{h,\ell,k}$ then there exists with high probability an allocation of the $m$
balls such that the maximal capacity of a bin is one; and if $n/m\geq
c$ with $c>c_{h,\ell,k}$ then such allocation does not exist with high
probability.


Our approach is completely different form the works cited above. We
consider the incidence graph of the $h$-uniform hypergraph $H$, which
is a bipartite graph $G=(A\cup B, E)$ where $A$ is the set of
hyperedges, i.e. are vertices in $G$ with degree $h$ and $B$ is the set of vertices of $H$. 
We then consider spanning subgraphs $S=(A\cup B, F)$ of $G$ with degree constraints:
any vertex from $A$ must have degree at most $\ell$ in $S$ while any
vertex from $B$ must have degree at most $k$ in $S$. Let the size of
such a spanning subgraph be the number of edges $|F|$ in $S$. The
following claim is easy to check and will be the basis of our
approach: $H$ is $(\ell,k)$-orientable if and only if all vertices in
$A$ have degree $\ell$ in any maximum spanning subgraph with degree
constraints $(\ell,k)$.
In this case, the size of any maximum spanning subgraph is
$\ell|A|$. Indeed in the case $(\ell,k)=(1,1)$ a spanning subgraph
with degree constraints $(1,1)$ is simply a matching of $G$. Based on this
observation, Bordenave, Salez and the author already derived the value
of $c^*_{h,1,1}$ in \cite{bls11}. The analysis of maximum spanning
subgraphs with general degree constraints $\ell$ and $k$ requires a
significant extension of the results in \cite{bls11}. Wagner
\cite{wag09} and, more closely related to our work, Salez \cite{salez}
study the generating polynomial for spanning subgraphs with degree
constraints. It follows from \cite{wag09} that for a sequence of
graphs whose size goes to infinity and having a random weak limit (see
definition in the sequel or \cite{besc01}, \cite{alst04}), the
rescaled size of a maximum spanning subgraph converges. In the case,
where the random weak limit is a Galton Watson tree, this limit is
characterised in \cite{salez} and computed in the particular case of
constant degree constraint.
In this work, we are able to simplify the characterisation of
\cite{salez} in our Proposition \ref{prop:MG} and to connect it to a simple message-passing algorithm. It allows us to bypass
the resolution of a difficult recursive distributional equation (which
was a key step in \cite{bls11} or \cite{salez}). 
It should perhaps be noted that the result in this paper is stronger
than \cite{csw07,fr07,dgmmpr10,fopa10,fkp11,gawo10} in the sense that it gives the size
of the largest spanning subgraph for all values of the parameter
$c>0$. We state this result explicitly in the following theorem:
\begin{theorem}\label{the:over}
Denote by $M_{\ell,k}(H)$ the size of a maximum spanning subgraph of
the bipartite graph $H$ with degree constraints $(\ell,k)$. With the
same notation as in Theorem \ref{the:hyperg} and for any integers
$h>\ell\geq 1$ and $k\geq 1$, we define the function of $(q,c)\in
[0,1]\times \RR_+$:
\BEAS
\Fcal_{\ell,k}(q,c) = \EE\left[ \min\left( \ell,
    \Bin(h,1-Q(chq,k))\right)\right]+\frac{kQ(chq,k+1)}{c}.
\EEAS
Then we have for any $c>0$,
\BEAS
\lim_{n\to \infty}\frac{1}{cn} M_{\ell,k}(H_{n,\lfloor cn \rfloor,h}) =
\inf_{q\in [0,1]}\Fcal_{\ell,k}(q,c).
\EEAS
\end{theorem}
It follows from calculations made in Section \ref{sec:orient} that for
$c<c^*_{h,\ell,k}$, we have $\inf_{q\in
  [0,1]}\Fcal_{\ell,k}(q,c)=\ell$ while for $c>c^*_{h,\ell,k}$, we
have $\inf_{q\in [0,1]}\Fcal_{\ell,k}(q,c)<\ell$. In particular, for
$c>c^*_{h,\ell,k}$, Theorem \ref{the:over} gives the asymptotic
fraction for the number of balls/keys that cannot be stored in the system.
A similar statement was proved in \cite{frpa09} in the particular case
of $\ell=k=1$ corresponding to spanning subgraphs being matchings. In
the sequel, we will give a more general statement of this result (see
Theorem \ref{the:limM}) which allows to deal with a larger family of
random hypergraphs.
Indeed,  we believe that this approach will allow to deal with more complex
situations where degree constraints could be random and/or asymptotic degree
distributions could be changed and can lead to efficient load
balancing algorithms. We refer to \cite{llm} where such extensions are
explored. 

The rest of this paper is organised as follows. We give an overview of our proof in Section \ref{sec:gov}.
We start with a careful analysis of a Boltzmann-Gibbs distribution on
spanning subgraphs with degree constraints in Section \ref{sec:gibbs}
and show some crucial monotonicity of the model. We then give an
explicit characterisation of the size of a maximum spanning subgraph
when the underlying graph is a finite tree and are able to extend it
to a possibly infinite tree using the important notion of
unimodularity \cite{aldouslyons}. In Section \ref{sec:comp}, we apply
these general results to the particular case where the underlying tree
is a branching process. This allows us to derive the asymptotic for
the size of a maximum spanning subgraph for a converging sequence of
bipartite graphs. In Section \ref{sec:orient}, we apply our results
for the particular sequence of graph $H_{n,\lfloor cn\rfloor, h}$ and
derive Theorem \ref{the:hyperg}.

\section{Gibbs measures and overview of the proof}\label{sec:gov}

We consider a finite simple graph $G=(V,E)$ with a vector of $\N^V$ denoted by $\bw=(w_v,\: v\in V)$ and called the vector of (degree) constraints. We are interested in spanning subgraphs $(V,F)$ with degree constraints given by the vector $\bw$. Each such subgraph is determined by its edge-set $F\subseteq E$ encoded by the vector $\bB = (B_e,\:e\in E)\in \{0,1\}^E$ defined by $B_e=1$ if and only if $e\in F$. We say that a spanning subgraph $\bB$ satisfies the degree constraints or is admissible if for all $v\in V$, we have $\sum_{e\in \d v}B_e\leq w_v$, where $\d v$ denotes the set of incident edges in $G$ to $v$.
We introduce the family of probability distributions on the set of admissible spanning subgraphs parametrised by a parameter $z>0$:
\BEA
\label{eq:gib}\mu_{G}^z(\bB) = \frac{z^{\sum_e B_e}}{P_G(z)},
\EEA
where $P_G(z) = \sum_{\bB} z^{\sum_e B_e}\prod_{v\in V}\ind(\sum_{e\in \d v}B_e\leq w_v)$.
We also define the size of the spanning subgraph by $|F|=\sum_eB_e$ and denote the maximum size by $M(G)= \max\{\sum_e B_e:\: \bB \mbox{ admissible}\}$. Those spanning subgraphs which achieve this maximum are called maximum spanning subgraphs. For any finite graph, when $z$ tends to infinity, the distribution $\mu^z_G$ converges to the uniform distribution over maximum spanning subgraphs. For an admissible spanning subgraph, the degree of $v$ in the subgraph is simply $\sum_{e\in \d v}B_e$. By linearity of expectation, the mean degree of $v$  under the law $\mu_G^z$ is $D_v^z:=\sum_{e\in \d v}\mu_G^z\left(B_e=1\right)$ so that we have
\BEA
\label{eq:limM}M(G) = \frac{1}{2} \sum_{v\in V}\lim_{z\to \infty}D_v^z.
\EEA

The main part of our work will be devoted to the computation of the limit on the right-hand side of (\ref{eq:limM}).
In the remaining part of this section, we give an informal description of the main steps.
The reader interested in the mathematical proof can skip the rest of this section and proceeds directly to Section \ref{sec:gibbs}.

Recall from the Introduction that we see hypergraph $H$ as bipartite graph $G=(A\cup B, E)$. Then, an hypergraph is $(\ell,k)$-orientable if and only if all vertices in $A$ have degree $\ell$ in any maximum spanning subgraph of the corresponding bipartite graph with degree constraints $(\ell,k)$.
Indeed in this case, we have for any $v\in A$, $\lim_{z\to \infty}D^z_{v}=\ell$ so that the size of a maximum spanning subgraph is $M(G)=\ell|A|$. Our main result will show that $\lim_{n\to \infty} \frac{M(G_n(c))}{|A_n|}=\inf_{q}\Fcal_{\ell,k}(q,c)$, where $G_n(c)=(A_n\cup B_n,E_n)$ is the bipartite graph associated to $H_{n,\lfloor cn \rfloor,h}$ and the function $\Fcal_{\ell,k}(q,c)$ is defined in Theorem \ref{the:over}.
In particular, when $\inf_q\Fcal_{\ell,k}(q,c)<\ell$, our result allows us to
conclude that the hypergraph is not $(\ell,k)$-orientable for such
values of $c$. In order to get the second half of the Theorem
\ref{the:hyperg}, when $\inf_q\Fcal_{\ell,k}(q,c)=\ell$, we need to show that
with high probability (as $n$ tends to infinity) there are actually no
vertices in $A_n$ with degree less than $\ell$ in a maximum spanning
subgraph (since the limit ensures only that the number of such
vertices is $o(n)$). For this part of the proof, we use a density
argument which relies on a combinatorial argument of
\cite{gawoarxiv10}. This is done in Section \ref{sec:orient}.

We now concentrate on the computation of the limit $\lim_{n\to \infty} \frac{M(G_n(c))}{|A_n|}$ using techniques from the objective method developed by Aldous and Steele \cite{alst04}.
A fundamental ingredient of the proof is the fact that the bipartite
graphs associated to $h$-uniform hypergraphs considered in this paper,
i.e. with $n$ vertices and $cn$ hyperedges are locally tree-like: with
high probability, there is no cycle in a ball (of fixed radius) around
a vertex chosen at random. It is then instructive to study maximum
spanning subgraphs when the underlying graph is a tree. 
Let first study the Gibbs measures defined by (\ref{eq:gib}) in the limit $z\to \infty$ in order to analyse maximum spanning subgraphs. When the underlying graph is a finite tree, we can use a more direct and algorithmic way that we now describe. 

To study the $(\ell,k)$-orientability of the hypergraph $H$ associated
to the bipartite graph $G$, the vector of degree constraints $\bw$
should be chosen such that $w_v=\ell$ for $v\in A$ and $w_v=k$ for
$v\in B$.
For simplicity, we assume here that the vector of degree constraints is constant so that all vertices have the same degree constraint say $w\geq 1$. Consider now the following message-passing algorithm forwarding messages in $\{0,1\}$ on the oriented edges of the underlying tree $G$ as follows: at each round, each oriented edge forwards a message, hence two messages are sent on each edge (one in each direction) at each round.
The message passed on the oriented edge $\ore=(u,v)$ is $0$ if the sum of the incoming messages to $u$ from neighbours different from $v$ in previous round is at least $w$ and the message is $1$ otherwise, i.e. if the sum of the incoming messages is strictly less than $w$.
Let $\bI_k\in \{0,1\}^{\orE}$ be the vector describing the messages sent on the oriented edges in $\orE$ at the $k$-th round of the algorithm. Denote by $\Pcal_G$ the action of the algorithm on the messages in one round so that $\bI_{k+1}=\Pcal_G(\bI_k)$. Assume that the algorithm is initialised with all messages set to one: $\bI_0=\bold{1}$. Figure \ref{fig:dogx} shows an example with $w=2$ for the messages exchanged for the three first rounds.
\begin{figure}[htb]
\begin{center}
\psfrag{Pcal}{$\bI_1=\mathcal{P}_G(\bold{1})$}
\psfrag{Pcal1}{$\bI_2=\mathcal{P}_G(\bI_1)$}
\psfrag{Pcal2}{$\bI_3=\mathcal{P}_G\circ\mathcal{P}_G(\bI_1)$}
\includegraphics[width=5cm]{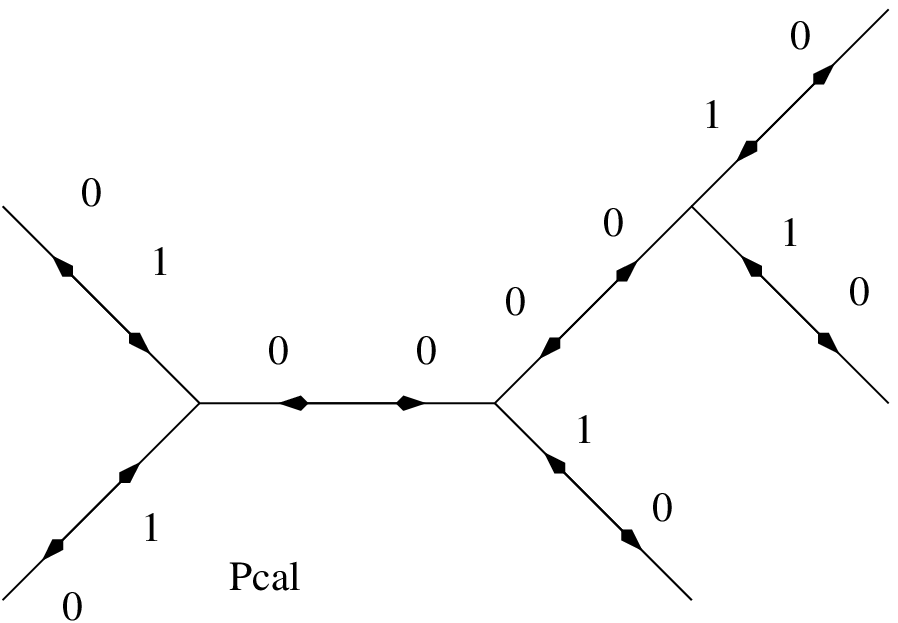}\hspace{0.7cm}
\includegraphics[width=5cm]{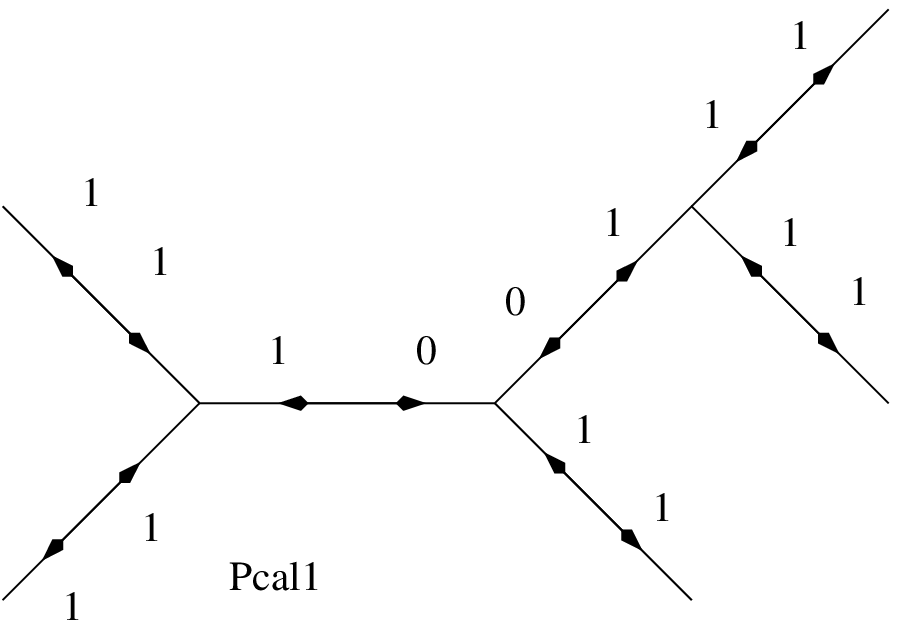}\hspace{0.7cm}
\includegraphics[width=5cm]{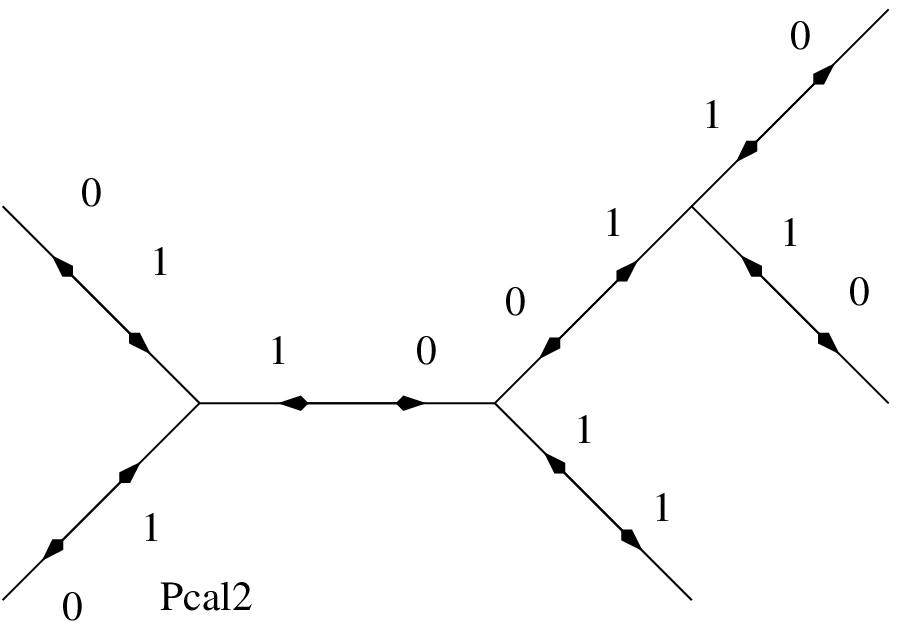}
\hspace{25pt} \caption{Iterating $\mathcal{P}_G$ on a finite tree
  (with $w_v=2$).}
\label{fig:dogx}
\end{center}\end{figure}
In Figure \ref{fig:dogx}, iterating a fourth time the algorithm would again give $\bI_3$. Indeed, it is easy to see that the algorithm will converge on any finite tree after a number of steps equals to at most the diameter of the tree, whatever the initial condition. Hence the messages of the algorithm converge to a vector $\bI^*=(I^*_{\ore}, \ore\in \orE)$ solving the fixed-point equation $\bI^*=\Pcal_G(\bI^*)$ and the size of a maximum spanning subgraph is given by
\BEA
\label{eq:algo}\frac{1}{2}\sum_{v\in V} \left(2w\ind \left(\sum_{\ore\in \partial v}I^*_{\ore}\geq w+1\right)+\ind \left(\sum_{\ore\in \partial v}I^*_{\ore}\leq w\right)\sum_{\ore\in \partial v}I^*_{\ore}\right),
\EEA
where $\partial v$ is the set of oriented edges toward $v$.
For example, one can check on Figure \ref{fig:dogx} using $\bI_3=\bI^*$ that the formula given by (\ref{eq:algo}) equals 5 which is the size of a maximum spanning subgraph.

Proving the correctness of the algorithm and of the formula (\ref{eq:algo}) on finite trees is simple once the following observations have been made:
\begin{itemize}
\item a vertex $v$ such that $\sum_{\ore\in \partial v}I^*_{\ore}\geq w+1$, will have degree $w$ in any maximum spanning subgraph and all messages in $\bI^*$ sent by $v$ will be zero.
\item an edge $(u,v)$ such that $I^*_{u\to v}=I^*_{v\to u}=1$ will be covered (i.e. with $B_{(u,v)}=1$) by any maximum spanning subgraph.
\item an edge $(u,v)$ such that $I^*_{u\to v}=I^*_{v\to u}=0$ will never be covered (i.e. with $B_{(u,v)}=0$) in any maximum spanning subgraph.
\end{itemize}

Note that the correctness of the algorithm is ensured for trees only,
but the definition of the algorithm does not require the graph to be a
tree. It makes only local computations and can be used on any graph
(without guarantee of converging and even in this case without
guarantee of correctness). Since the bipartite graphs associated to
$H_{n,\lfloor cn\rfloor,h}$ are not trees but are locally tree like,
it is tempting to use the algorithm directly on these graphs. It turns
out that for low values of $c$, the algorithm will converge and will
also be correct (with high probability). This is not a surprise since
for $c<1/h$, the random graph is essentially a collection of small
tree components. It turns out that the algorithm allows to compute the
size of a maximum spanning subgraph for values of $c$ above $1/h$ but
it breaks down at some higher value of $c$, indeed exactly when the
$(\ell, k+1)$-core as defined in \cite{gawo10} appears. We did not try
to make this claim rigorous as it is not required for our analysis. 
In the particular case where $h=2$ and $\ell=k=1$ (i.e. $w_v=1$ for
all vertices $v$), our problem reduces to
the problem of maximum matching in Erd\H{o}s-R\'enyi random graphs and
our claim follows from the analysis of Karp and Sipser \cite{ks81}.
More precisely, the greedy algorithm described above corresponds in
this case to the standard leaf-removal algorithm studied in
\cite{ks81} (see also Proposition 12 in \cite{blsaop11} and the
Appendix of \cite{bls11}). The $(1,1)$-core is simply called the
core in these references and it appears if the mean degree is above $e$.
If there exists a $(\ell,k+1)$-core, the intuition is as follows: if we
approximate the graph by a branching process and run the algorithm on
this tree starting from level $k$ from a root, the influence of the
boundary conditions on the value given by the algorithm at the root is
positive as we let $k$ tends to infinity. This translates into the fact
that on an infinite tree, there might exist several solutions to the
fixed-point equation $\bI=\Pcal_G(\bI)$. It is then natural to ask
which one is associated to the large $n$ limit maximum spanning subgraphs.
From an algorithmic viewpoint, there is 'no correlation decay' and the
computations made by the algorithm is not anymore local.

In order to bypass this absence of 'correlation decay', we borrow
ideas from statistical physics by introducing the Gibbs measures
$\mu_G^z$ defined in (\ref{eq:gib}) parametrised by a parameter $z>0$
(usually called the activity or the fugacity) \cite{demo10}. 
Informally, the introduction of this parameter $z$ will allow us
to capture sufficient additional information on our problem in order
to identify the 'right' solution to the fixed-point equation
$\bI=\Pcal_G(\bI)$, when we let $z$ goes to infinity.
Our first step in the analysis of these measures is to derive a
message-passing algorithm allowing to compute the mean degree $D^z_v$
of any vertex $v$ in a spanning subgraph taken at random according to
the probability distribution $\mu_G^z$. We will proceed by first
defining the local computations required at each node and we call them
the local operators. We use these building blocks to define a
message-passing algorithm which is valid on any finite tree. In
particular, we show that as $z$ tends to infinity, the dynamic of the
algorithm becomes exactly the one described previously in this
section. As a by-product, we prove the validity of (\ref{eq:algo}),
see Proposition \ref{prop:MGfinite}.
The advantage of considering these algorithms with $z<\infty$ is that
we are able to define them properly on infinite graphs if these graphs
have a natural stationarity property called unimodularity
\cite{aldouslyons}. Section \ref{sec:unimod} presents this notion in
details and shows how it is used in our framework. Proposition
\ref{prop:fpoint} shows that 'the message-passing algorithm converges
to a unique fixed point for any $z<\infty$'.
In other words, Proposition \ref{prop:fpoint} shows that for any
$z<\infty$, there is 'correlation decay'. In the limit
$z=\infty$, the algorithm may 'have more than one fixed-point'
(corresponding to the absence of correlation decay mentioned above),
but Proposition \ref{prop:MG} allows to select the 'valid' fixed-point
which gives the mean degree of a vertex 'picked at random' in a
'maximum' spanning subgraph of a possibly infinite graph. The notion
of unimodularity is a key concept in making these statements
rigorous. This is the price to pay in order to work directly with the
infinite, probabilistic object obtained as the limit of our finite
problem as $n$ tends to infinity. The reward of this objective method
\cite{alst04} is that in our case, this infinite object with fixed
distributional properties is simple to analyse (see Section
\ref{sec:comp}) and captures all the necessary information on the
asymptotic behaviour of the original sequence in order to compute the
limit (\ref{eq:limM}).

\section{Analysis of Gibbs measures}\label{sec:gibbs}

This section is devoted to the computation of the limit appearing on the right-hand side of (\ref{eq:limM}). We first deal with the particular case where $G$ is a star (i.e. a tree with one internal node and several leaves) then we show that the computation can be done recursively on a finite tree. We also show how these recursions can be extended to possibly infinite unimodular networks (defined in the sequel).

\subsection{Local operators}

In this section, we define local operators associated to the degree constraints defining an admissible spanning subgraph. In particular, the set $E$ used in this section will be interpreted later as the set of edges incident to a given vertex.

Let $E$ be a finite set of elements $e$ and $\Yb=(Y_e,\:e\in E)\in [0,\infty)^E$. We define a probability measure $\mu$ on $\{0,1\}^E$ as follows: the binary random variables $B_e$ with $e\in E$ are independent Bernoulli random variables with $\mu(B_e=1) =\frac{Y_e}{1+Y_e}\in [0,1)$. For a given integer $w$, we will be mainly interested in the measure obtained from $\mu$ by conditioning on $\sum_{e\in E}B_e\leq w$ and we denote it by $\omu$. 
Simple calculations show that for any $(\eta(e),\:e\in E)\in \{0,1\}^E$ such that $\sum_{e\in E}\eta(e)\leq w$, we have
\BEAS
\omu\left(B_e=\eta(e):\: \forall e\in E\right):=\mu\left( B_e=\eta(e):\: \forall e\in E \:\middle\vert\: \sum_{e\in E}B_e\leq w\right) &=& \frac{\prod_{e\in E}Y_e^{\eta(e)}}{\sum_{S\subset E,|S|\leq w} \Yb^S},
\EEAS
where for any set $S\subset E$, $\Yb^S$ is a convenient notation for $\Yb^S:=\prod_{e\in S}Y_e$ and $\Yb^\emptyset =1$ and $|S|$ is the number of elements in $S$. Note that with $Y_e=z$ for all $e$, we recover (\ref{eq:gib}) for a simple star graph, where all edges have exactly one node in common with associated degree constraint $w$ and all other nodes have constraints larger than one.

For a given $e\in E$, we introduce the following notations:
\BEAS
\omu\left(B_e=1\right)=\mu\left(B_e=1\:\middle\vert\: \sum_{e\in E}B_e\leq w\right)
&=&\frac{Y_e \Rcal_e(\Yb)}{1+Y_e \Rcal_e(\Yb)},
\EEAS
where (we use the notation $E\backslash e$ for $E\backslash \{e\}$):
\BEA
\label{eq:defRe}
\Rcal_e(\Yb) := \frac{\sum_{S\subset E\backslash e,\:|S|\leq w-1} \Yb^S}{\sum_{S\subset E\backslash e,\:|S|\leq w} \Yb^S}\leq 1.
\EEA
Note in particular that the function $\Yb\mapsto\Rcal_e(\Yb)$ depends only on the components $Y_\ell$ with $\ell\neq e$.

We also define
\BEAS
\Dcal(\Yb) := \sum_{e\in E}\mu\left(B_e=1\:\middle\vert\: \sum_{e\in E}B_e\leq w\right)=\sum_{e\in E}\frac{Y_e \Rcal_e(\Yb)}{1+Y_e \Rcal_e(\Yb)}.
\EEAS 
With the interpretation given above of the star graph with all $Y_e=z$, $\Dcal(\Yb)$ is simply the mean degree of the central node under the probability distribution (\ref{eq:gib}).

The following proposition shows crucial monotonicity properties of the functions introduced above. It follows the argument in \cite{salez} (see Lemma 1 and 2) using results from the theory of negative dependence in \cite{pe00}.

\begin{proposition}\label{prop:localmon}
The mapping $\Rcal_e:[0,\infty)^{E\bac e}\to (0,1]$ is non-increasing in each variable.
The mapping $\Dcal:[0,\infty)^E\to [0,|E|)$ is strictly increasing in each variable.
The mapping $z\in [0,\infty)\mapsto z\Rcal_e(z\Xb)$ (where $z\Xb$ denotes the vector with entries $zX_e$) is strictly increasing if $\bX>0$.
\end{proposition}
\begin{proof}
From Theorem 2.7 in \cite{pe00}, we know that the measure $\omu$ is negatively correlated: for $e\neq f$,
\BEAS
\omu(B_e=1, B_f=1)\leq \omu(B_e=1)\omu(B_f=1).
\EEAS
This result directly gives the first point of the proposition since we have:
\BEAS
\frac{\d \Rcal_e(\Yb)}{\d Y_f} &=& \frac{\sum_{S\subset E\backslash\{e,f\},\:|S|\leq w-2} Y^S}{\sum_{S\subset E\backslash\{e\},\:|S|\leq w} Y^S}-\frac{\left(\sum_{S\subset E\backslash\{e\},\:|S|\leq w-1} Y^S\right)\left(\sum_{S\subset E\backslash\{e,f\},\:|S|\leq w-1} Y^S\right)}{\left(\sum_{S\subset E\backslash\{e\},\:|S|\leq w} Y^S\right)^2}\\
&=& \frac{\omu(B_e=0)\omu(B_e=1,B_f=1)-\omu(B_e=1)\omu(B_e=0,B_f=1)}{Y_eY_f\omu(B_e=0)^2}\\
&=& \frac{\omu(B_e=1,B_f=1)-\omu(B_e=1)\omu(B_f=1)}{Y_eY_f\omu(B_e=0)^2}\leq 0.
\EEAS
For the second point, we compute the derivative of $\Dcal$ as follows:
\BEAS
\frac{\d \Dcal(\Yb)}{\d Y_f} &=& \frac{\Rcal_f(\Yb)}{(1+Y_f\Rcal_f)^2}+\sum_{e\in E\backslash f}\frac{Y_e}{(1+Y_e\Rcal_e(\Yb))^2}\frac{\d \Rcal_e}{\d Y_f}(\Yb)\\
&=& \frac{1}{Y_f}\sum_{e\in E}\left(  \omu(B_e=1,B_f=1) - \omu(B_e=1)\omu(B_f=1)\right).
\EEAS
We need to prove that this last quantity is positive.
Let $\omu_k$ denote the law $\mu$ conditioned on $\sum_{e\in
  E}B_e=k$. By Theorem 2.7 in \cite{pe00}, this measure is still
negative correlated so that for $e\neq f$, we have:
\BEAS
\omu_k(B_e=1,B_f=0) = \omu_k(B_e=1)-\omu_k(B_e=1,B_f=1)\geq \omu_k(B_e=1)\omu_k(B_f=0).
\EEAS
In particular, we get
\BEAS
\omu_k(B_e=1|B_f=0)\geq \omu_k(B_e=1)\geq \omu_k(B_e=1|B_f=1).
\EEAS
Add an extra variable $B_{f^*}$, so that $\omu_k(B_e=1|B_{f^*}=0)=\mu\left(B_e=1\middle\vert \sum_{f\in E} B_f=k\right)$ and $\omu_k(B_e=1|B_{f^*}=1)=\mu\left(B_e=1\middle\vert \sum_{f\in E} B_f=k-1\right)$. Hence we proved
\BEAS
\mu\left(B_e=1\middle\vert \sum_{f\in E} B_f=k\right)\geq\mu\left(B_e=1\middle\vert \sum_{f\in E} B_f=k-1\right).
\EEAS 
Denoting $a_k=\mu\left(B_e=1, \sum_{f\in E} B_f=k\right)$ and
$b_k=\mu\left(\sum_{f\in E} B_f=k\right)$, previous inequality shows that $\frac{a_k}{b_k}$ is non-decreasing in $k$ so that for $1\leq k\leq w$, we have $\frac{\sum_{j\leq k-1}a_j}{\sum_{j\leq k-1}b_j}\leq \frac{a_k}{b_k}\leq \frac{\sum_{w\geq j\geq k}a_j}{\sum_{w\geq j\geq k}b_j}$ which gives
\BEAS
\omu\left(B_e=1\middle\vert \sum_{f\in E} B_f\leq k-1\right)\leq \omu\left(B_e=1\middle\vert \sum_{f\in E} B_f\geq k\right),
\EEAS
so that we get 
\BEAS
\omu\left(B_e=1, \: \sum_{f\in E} B_f\geq k\right)\left(1-\omu\left( \sum_{f\in E} B_f\geq k\right)\right)&\geq&\omu\left(\sum_{f\in E} B_f\geq k\right)\omu\left(B_e=1,\: \sum_{f\in E} B_f\leq k-1\right)\\
\omu\left(B_e=1, \: \sum_{f\in E} B_f\geq k\right)&\geq&
\omu\left(\sum_{f\in E} B_f\geq k\right)\omu\left( B_e=1\right)\\
\omu\left(B_e=1, \: \sum_{f\in E} B_f\geq k\right)\left(1-\omu\left( B_e=1\right)\right)&\geq&
\omu\left(\sum_{f\in E} B_f\geq k, \:B_e=0\right)\omu\left( B_e=1\right)\\
\omu\left( \sum_{f\in E} B_f\geq k\middle\vert B_e=1\right)&\geq&\omu\left( \sum_{f\in E} B_f\geq k\middle\vert B_e=0\right),
\EEAS
in particular, we get $\omu\left( \sum_{f\in E} B_f\geq k\middle\vert B_e=1\right)\geq\omu\left( \sum_{f\in E} B_f\geq k\right)$. Note that for $k=0$, the inequality is strict so that summing over $k$, we get
\BEAS
\sum_{f\in E}\omu(B_f=1\vert B_e=1)>  \sum_{f\in E} \omu(B_f=1),
\EEAS
which concludes the proof for $\Dcal$.
Finally, we compute:
\BEAS
\frac{\d z\Rcal_e(z\Xb)}{\d z} &=& \Rcal_e(z\Xb)+z\sum_{f\in E\backslash e}X_f\frac{\d \Rcal_e(z\Xb)}{\d X_f}\\
&=& \frac{\omu^z(B_e=1)}{z X_e\omu^z(B_e=0)}+\sum_{f\in E\backslash e}\frac{\omu^z(B_e=1,B_f=1) - \omu^z(B_e=1)\omu^z(B_f=1)}{zX_e\omu^z(B_e=0)^2}
\EEAS
where $\omu^z$ is the conditioned measure obtained with parameter $\Yb=z\Xb$.
Then, we have
\BEAS
\frac{\d z\Rcal_e(z\Xb)}{\d z} &=&\frac{1}{z X_e\omu^z(B_e=0)^2}\left( \sum_{f\in E}\omu^z(B_f=1,B_e=1)- \omu^z(B_f=1)\omu^z(B_e=1)\right).
\EEAS
It follows again from Theorem 2.7 in \cite{pe00} that $\omu^z$ is still negatively correlated (stability under external fields) so that previous calculation is still valid and this concludes the proof.
\end{proof}

In what follows, we will need to consider the extension of previous mappings in order to cover the case $Y_e=\infty$. We now consider $\Yb\in [0,\infty]^E$.
Let $E'(\Yb)\subset E$ be the set of elements $e$ such that $Y_e=\infty$. 
We now show that it is possible to extend continuously the mappings $\Rcal_e$ and $\Dcal$. To start, it is easy to check that for any extension of $\omu(.)$, we must have:
\begin{itemize}
\item if $|E'(\Yb)|\geq w$, then $\sum_{e\in E'(\Yb)}\omu\left( B_e=1\right)=w$, so that $\forall e\notin E'(\Yb)$, we have for any $Y_e\geq 0$, $\omu(B_e=0)=1$;
\item if $|E'(\Yb)|< w$, then for any $Y_e\geq 0$ with $e\notin E'(\Yb)$, we have for any $(\eta(e),\:e\in E)\in\{0,1\}^E$,
\BEA
\label{eq:omuw}\omu\left( B_e=\eta(e):\: e\in E \right) &=&
\ind\left(\prod_{e\in E'(\Yb)}\eta(e)=1 \right)\frac{\prod_{e\in E\backslash E'(\Yb)}Y_e^{\eta(e)}}{\sum_{S\subset E\backslash E'(\Yb),|S|\leq w-|E'(\Yb)|} \Yb^S}.
\EEA
\end{itemize}

This simple fact allows to extend the mapping $\Rcal_e$ as follows:
\begin{lemma}
We define the mapping $\Rcal_e:[0,\infty]^{E\bac e}\to [0,1]$ defined by (\ref{eq:defRe}) on $[0,\infty)^E$ and on $[0,\infty]^{E \bac e}\bac [0,\infty)^{E\bac e}$, we let $E'_e(\Yb)=\{f\in E\bac e,\: Y_f=\infty\}$ and define $\Rcal_e$ by
\begin{itemize}
\item if $|E'_e(\Yb)|< w$, then
\BEAS
\Rcal_e(\Yb) := \frac{\sum_{S\subset E\bac E'_e(\Yb)\cup\{e\},\:|S|\leq w-|E'_e(\Yb)|-1} \Yb^S}{\sum_{S\subset E\bac E'_e(\Yb)\cup\{e\},\:|S|\leq w-|E'_e(\Yb)|} \Yb^S}\in (0, 1].
\EEAS
\item if $|E'_e(\Yb)|\geq w$, then $\Rcal_e(\Yb)=0$.
\end{itemize}
In particular, we have
\BEA
\label{eq:Rcal0}\Rcal_e(\Yb) =0 &\Leftrightarrow& \sum_{f\in E\bac e}\ind(Y_f=\infty)\geq w.
\EEA
The mapping $\Rcal_e$ is continuous.
\end{lemma}
\begin{proof}
If $|E'_e(\Yb)|\geq w$, then we have for all $Y_e>0$, $\omu(B_e=0) = \frac{1}{1+Y_e\Rcal_e(\Yb)} =1$ so that $\Rcal_e(\Yb)=0$.

If $|E'_e(\Yb)|< w$, then from (\ref{eq:omuw}) we have for all $Y_e\in(0,\infty)$,
\BEAS
\omu(B_e=0) &=&\sum_{U\subset E\bac E'(\Yb),\:|U|\leq w-|E'(\Yb)|, \:e\notin U}\frac{\Yb^U}{\sum_{S\subset E\bac E'(\Yb),\:|S|\leq w-|E'(\Yb)|} \Yb^S}\\
&=& \frac{\sum_{S\subset E\bac E'_e(\Yb)\cup\{e\},\:|S|\leq w-|E'_e(\Yb)|} \Yb^S}
{\sum_{S\subset E\bac E'_e(\Yb)\cup\{e\},\:|S|\leq w-|E'_e(\Yb)|} \Yb^S+Y_e\sum_{S\subset E\bac E'_e(\Yb)\cup\{e\},\:|S|\leq w-|E'_e(\Yb)|-1} \Yb^S}.
\EEAS
And the lemma follows since $\omu(B_e=0) = \frac{1}{1+Y_e\Rcal_e(\Yb)}$.
\end{proof}

It is also possible to extend the mapping $\Dcal$ as follows:
\begin{lemma}\label{lem:Dcal}
The mapping $\Dcal:[0,\infty]^E\to [0,|E|]$ defined by
\BEAS
\Dcal(\Yb) &=& \sum_{e\in E}\frac{Y_e\Rcal_e(\Yb)}{1+Y_e\Rcal_e(\Yb)}\ind(Y_e<\infty) +w\wedge \sum_{e\in E} \ind\left( Y_e=\infty\right),
\EEAS
is continuous.
\end{lemma}
\begin{proof}
Since $\Rcal_e(\Yb)\in [0,1]$, we need only to deal with the case where there exists an element $e$ with $Y_e=\infty$ and $\Rcal_e(\Yb)=0$. In this case we have $\sum_{f\in E\bac e}\ind(Y_f=\infty)\geq w$, so that $E'(\Yb)=\{f\in E,\:Y_f=\infty\}$ is such that $|E'(\Yb)|\geq w+1$ so that $\Rcal_f(\Yb)=0$ for all $f\in E$. In particular, we have
\BEAS
\sum_{e\in E}\frac{Y_e\Rcal_e(\Yb)}{1+Y_e\Rcal_e(\Yb)}\ind(Y_e<\infty) &=& 0\mbox{ and,}\\
\Dcal(\Yb)=\sum_{e\in E}\frac{Y_e\Rcal_e(\Yb)}{1+Y_e\Rcal_e(\Yb)}&=&\sum_{e\in E}\frac{Y_e\Rcal_e(\Yb)}{1+Y_e\Rcal_e(\Yb)}\ind(Y_e=\infty)= \sum_{e\in E'(\Yb)}\omu(B_e)=w,
\EEAS
and the lemma follows.
\end{proof}

A third operator will be important in the sequel. By Proposition \ref{prop:localmon}, we can define for $\Xb\in(0,1]^E$:
\BEAS
\Qcal_e(\Xb) &=& \lim_{z\to \infty} \uparrow z\Rcal_e(z\Xb)= \frac{\sum_{S\subset E\backslash e,\:|S|= w-1} \Xb^S}{\sum_{S\subset E\backslash e,\:|S|= w} \Xb^S}.
\EEAS

Since the mapping $(z,\Xb)\in (0,\infty)\times (0,1]^E\mapsto z\Rcal_e(z\Xb)$ is non-decreasing in $z$ and non-increasing in $\Xb$, we can extend the mapping $\Qcal_e:[0,1]^E\to [0,\infty]$ continuously by,
\BEA
\label{eq:Qcal0}\Qcal_e(\Xb) = \infty &\Leftrightarrow& \sum_{f\in E\backslash e} \ind(X_f>0)<w.
\EEA

\begin{lemma}\label{lem:Q}
The mapping $\Qcal_e:[0,1]^E\to [0,\infty]$ is non-increasing in $
\Xb$ and for any $\Xb\in [0,1]^E$, we have
\BEAS
\sum_{e\in E}\frac{X_e\Qcal_e(\Xb)}{1+X_e\Qcal_e(\Xb)} \ind(\Qcal_e(\Xb)<\infty) &=& w\ind\left(\sum_{e\in E}\ind(X_e>0)\geq w+1 \right).
\EEAS
\end{lemma}
\begin{proof}
If $\forall e$, either $\Qcal_e(\Xb)=\infty$, or $X_e=0$, then the left-hand side is 0 and either there exists $e$ such that $\Qcal_e(\Xb)=\infty$, so that by (\ref{eq:Qcal0}) the right-hand side also equals 0, or for all $e$, $\Qcal_e(\Xb)<\infty$ and $X_e=0$, in which case $\sum_{e\in E}\ind(X_e>0)=0$, so that the right-hand side also equals 0.
Conversely, if there exists $e$ such that $\Qcal_e(\Xb)<\infty$ and $X_e>0$, then we have $\sum_{f\in E}\ind(X_f>0)\geq w+1$ so that $\Qcal_f(\Xb)<\infty$ for all $f\in E$. Hence we have
\BEAS
\sum_{e\in E}\frac{X_e\Qcal_e(\Xb)}{1+X_e\Qcal_e(\Xb)} \ind(\Qcal_e(\Xb)<\infty) &=&\ind\left(\sum_{e\in E}\ind(X_e>0)\geq w+1 \right)\sum_{e\in E}\frac{X_e\Qcal_e(\Xb)}{1+X_e\Qcal_e(\Xb)}\\
&=& \ind\left(\sum_{e\in E}\ind(X_e>0)\geq w+1 \right)\sum_{e\in E}
\frac{\sum_{|S|= w, e\in S} \Xb^S}{\sum_{|S|= w} \Xb^S}\\
&=& w\ind\left(\sum_{e\in E}\ind(X_e>0)\geq w+1 \right).
\EEAS
\end{proof}

In view of (\ref{eq:Rcal0}) and (\ref{eq:Qcal0}), we now introduce a fourth important operator: for any $\bI\in \{0,1\}^E$ and any $e\in E$, we define
\BEA
\label{eq:defP}\Pcal_e(\bI) = \ind\left( \sum_{f\in E\bac e}I_f<w\right),\mbox{ where a sum over the empty set is equal to zero.}
\EEA
Clearly this operator is non-increasing in $\bI$.
The local operators defined in this section will be the building blocks of a global operator defined on the graph $G$. When the graph $G$ is a tree, there is a simple connection between this global operator and the distribution (\ref{eq:gib}) and we describe it in the next section. Before that, we introduce some notations. In the sequel, there will be more than one set of element $E$ and weight $w$ involved. Indeed a local operator will be associated to each vertex or (directed) edge of $G$. To avoid ambiguity, we denote explicitly the dependence in $E,w$ as follows:
$\omu^{(E,w)}, \Rcal_e^{(E,w)}, \Dcal^{(E,w)}, \Qcal_e^{(E,w)},\Pcal_e^{(E,w)}$.

\subsection{Finite graphs}

We now come back to the case where $G=(V,E)$ is a finite simple graph and $\bw=(w_v,\: v\in V)$ is a vector of constraints.
For any $i\in V$, recall that $\d i$ denotes the set of edges incident to $i$. We now define local operators with associated set of elements $\d i$, constraint $w_i$ and some weights $\bY$. Clearly each edge $(ij)\in E$ will appear in the local operators associated to $\d i$ and $\d j$ with possibly different weights. In order to distinguish between these two cases, we will add an orientation to each edge: for the operator with  constraint $w_i$ , the set of elements will be the oriented edges adjacent to $i$ and directed towards $i$. 
We introduce now some
notations. We denote by $\orE$ the set of oriented edges.
For $i,j\in V^2$ two adjacent nodes, we denote by $i\to j$ the
oriented edge from $i$ to $j$. 
We still denote by $\partial i$ the set of oriented edges towards $i$. It will be clear from the context if we consider oriented or non-oriented edges when we write $\d i$.
We also define $\partial i \backslash j$ as the set of oriented edges towards $i$ except $j\to i$.
For a vector $\bX\in \RR^E$ (or $\RR^{\orE}$) and a subset of edges $F\subset E$ (or oriented edges $F\subset \orE$), we denote by
$\bX[F]$ the vector restricted to $F$: $\bX[F]=(X_i,\:i\in F)$. 

If the graph $G$ is finite and acyclic, it is easy to see that the law of
each marginal $\bB[\d  i]$ under $\mu_G^z$ is exactly described by
previous local probability measures $\omu^{(\d i, w_i)}$, with constraint: $w_i$; set of elements: the oriented edges toward $i$, $\d i$; and for some parameters $(Y_{\ore}(z), \ore \in \d i)$ to be computed. 
We now describe how to compute these parameters: $\bY(z)=(Y_{\ore}(z),
\ore\in \orE) \in \RR^{\orE}$. We first define a global map $\Rcal: \bY\in\RR^{\orE} \mapsto  \bZ\in\RR^{\orE}$.
For a given oriented edge $j\to i$, we consider the local map $\Rcal^{(\d i,w_i)}_{j\to
  i}(.)$ as in previous section with set of elements $\d i$, and
constraint $w_i$ so that
\BEAS
\Rcal^{(\d i,w_i)}_{j\to i}(\Yb) = \frac{\sum_{S\subset \d i\backslash j,\:|S|\leq
    w_i-1} \Yb^S}{\sum_{S\subset \d i\backslash j,\:|S|\leq w_i} \Yb^S}.
\EEAS
In order to define $\bZ=\Rcal_G(\bY)$, we set $Z_{i\to j} = \Rcal^{(\d i,w_i)}_{j\to
  i}(\Yb)$. Note that the notations are consistent with previous section and $Z_{i\to j} = \Rcal^{(\d i,w_i)}_{j\to
  i}(\Yb)$ is a function of the $Y_{\ell \to i}$ for $\ell\neq j$ neighbour of $i$.
In order to lighten the notation, when no ambiguity is possible, we will simply denote $\Rcal^{(\d i,w_i)}_{j\to
  i}$ by $\Rcal_{j\to i}$:\\
\begin{displaymath}
\xymatrix{
\ell_1 \ar[drr]^{Y_{\ell_1\to i}}&&&&&\\
\ell_2\ar[rr]^{Y_{\ell_2\to i}}&&i\ar[rrr]^-{Z_{i\to j} = \Rcal_{j\to i}(\Yb)}&&&j\\
\ell_3\ar[urr]_{Y_{\ell_3\to i}}&&&&&
}
\end{displaymath}

\begin{proposition}\label{prop:Recfinite}
On a finite acyclic graph $G$ and for any $z>0$, there is a unique solution to the
fixed point equation:
\BEA
\label{eq:fixpoint}\bY(z) = z\Rcal_G(\bY(z)).
\EEA
Moreover the marginal law of $\bB[\d i]$ under $\mu_G^z$ is given by the local probability measure $\omu^{(\d i, w_i)}$, with constraint $w_i$, set of elements $\d i$ and parameters $(Y_{\ore}(z), \ore \in \d i)$ being the restriction to coordinates in $\d i$ of the solution to (\ref{eq:fixpoint}).
\end{proposition}

The fact that the recursion (\ref{eq:fixpoint}) is exact on trees is a standard result in the literature on graphical models \cite{mm09} (in this context see Lemma 3 in \cite{salez}).
It follows directly from this proposition, that the mean degree of vertex $v\in V$ under $\mu_G^z$ defined by (\ref{eq:gib}) is given (when $G$ is a tree) by:
\BEA\label{eq:defD}
\Dcal_v(\bY(z)) = \sum_{\ore\in \d v}\frac{Y_{\ore}(z) \Rcal_{\ore}(\bY(z))}{1+Y_{\ore}(z) \Rcal_{\ore}(\bY(z))}.
\EEA

Recall that we are interested in the limit $z\uparrow \infty$ in which case, $\mu^z_G$ converges to the uniform distribution over maximum spanning subgraphs.
Iterating (\ref{eq:fixpoint}), we obtain $\bY(z) = z\Rcal_G(z\Rcal_G(\bY(z)))$. It is natural to define on $[0,1]^{\orE}$,
\BEAS
\Qcal_G(\Xb)&=&\lim_{z\to \infty} \uparrow z\Rcal_G(z\Xb),
\EEAS
which is well defined by the monotonicity properties shown in previous section. Then, $\Xb'=\Qcal_G(\Xb)$ is defined by:
\BEAS
X'_{i\to j}=\Qcal_{j\to i}(\Xb) &=&\frac{\sum_{S\subset \partial i \backslash j,\:|S|= w_i-1} \Xb^S}{\sum_{S\subset \partial i \backslash j,\:|S|= w_i} \Xb^S}.
\EEAS

As will be shown latter in a more general context, we have $ \lim_{z\to \infty}\uparrow \bY(z)=\bY \in [0,\infty]^{\orE}$ so that by passing to the limit in $\bY(z) = z\Rcal_G(z\Rcal_G(\bY(z)))$, we see that $\bY =\Qcal_G\circ \Rcal_G(\bY)$.

Recall that $\Pcal_e$ was defined in (\ref{eq:defP}) and we extend it to a global operator: $\Pcal_G:\{0,1\}^{\orE}\to \{0,1\}^{\orE}$ as follows:
\BEAS
\forall \bI\in \{0,1\}^{\orE}, \Pcal_{j\to i}(\bI) = \ind\left( \sum_{f\in \d i\bac j}I_f<w_i\right).
\EEAS

\begin{proposition}\label{prop:MGfinite}
Let $\bI\in \{0,1\}^{\orE}$ be defined by $I_{\ore}=\ind(Y_{\ore}=\infty)$.
For a finite acyclic graph $G$, $\bI$ is the unique solution to the fixed point equation: $\bI =\Pcal_G(\bI)$ and we have
\BEA
\label{eq:MGfinite}2M(G)=\lim_{z\to \infty}\sum_{v\in V}\Dcal_v(\bY(z)) &=& \sum_{v\in V}\left( w_v\ind\left(\sum_{\ore\in \d v} I_{\ore}\geq w_v+1\right)+w_v\wedge \sum_{\ore \in \d v} I_{\ore}\right).
\EEA
\end{proposition}
\begin{proof}
If $\bX=\Rcal_G(\bY)$ and $\bI'$ is defined by $I'_{\ore} = \ind(X_{\ore}>0)$, then by (\ref{eq:Rcal0}) and (\ref{eq:Qcal0}), we see that $\bI=\Pcal_G(\bI')$ and $\bI'=\Pcal_G(\bI)$ so that $\bI=\Pcal_G\circ\Pcal_G(\bI)$. Moreover, starting form the leaves and iterating $\Pcal_G$, is is easily seen that $\bI'=\bI$ and that it is the unique solution to $\bI=\Pcal_G(\bI)$.
Then, we have
\BEAS
\lim_{z\to \infty}\sum_{v\in V}\Dcal_v(\bY(z))
&=& \sum_{v\in V}\left(\sum_{\ore\in \d v}\frac{Y_{\ore}\Rcal_{\ore}(\bY)}{1+Y_{\ore}\Rcal_{\ore}(\bY)}\ind(Y_{\ore}<\infty)+w_v\wedge \sum_{\ore\in \d v} I_{\ore} \right).
\EEAS
Using $X_{-\ore}=\Rcal_{\ore}(\Yb)$ and $Y_{\ore}=\Qcal_{-\ore}(\Xb)$, we get
\BEAS
\lim_{z\to \infty}\sum_{v\in V}\Dcal_v(\bY(z))
&=&\sum_{v\in V}\left(\sum_{\ore\in \d v}\frac{X_{\ore}\Qcal_{\ore}(\bX)}{1+X_{\ore}\Qcal_{\ore}(\bX)}\ind(\Qcal_{\ore}(\bX)<\infty)+w_v\wedge \sum_{\ore\in \d v} I_{\ore} \right)\\
&=& \sum_{v\in V}\left(w_v\ind\left( \sum_{\ore\in \d v }\ind(X_{\ore}>0)\geq w_v+1\right)+w_v\wedge \sum_{\ore\in \d v} I_{\ore} \right),
\EEAS
where we used Lemma \ref{lem:Q} to get the last equality. The proposition now follows from the observation noted above that $\ind(X_{\ore}>0) =I'_{\ore}=I_{\ore}$.
\end{proof}
\begin{remark}
If we remove the assumption that the graph is acyclic, then
\begin{itemize}
\item there might be no solution to the fixed point equation $\bI=\Pcal_G(\bI)$. This is the case for the complete graph with $3$ vertices and all weights equal to $1$.
\item there might exist several solutions to the fixed point equation $\bI=\Pcal_G(\bI)$. This is the case for the complete graph with $4$ vertices and all constraints equal to $2$: all solutions can be obtained by permuting the role of the edges in the following picture:
\begin{figure}[htb]
\begin{center}
\includegraphics[width=2cm]{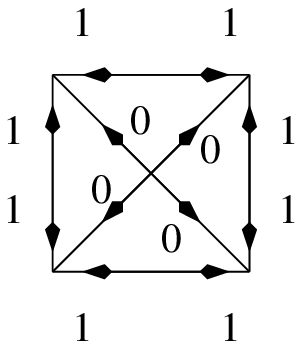}
\caption{solutions to $\bI=\Pcal_G(\bI)$.}
\label{fig:K4}
\end{center}
\end{figure}
\end{itemize}
\end{remark}

If $G$ is a finite tree, this proposition with (\ref{eq:limM}) shows that the size of a maximal spanning subgraph $M(G)$ boils down to the computation of (\ref{eq:MGfinite}) which is exactly the formula (\ref{eq:algo}) (for the particular case $w_v=w$).
We will now show how to extend this computation to the case where $G$ is possibly infinite.

\subsection{Infinite unimodular networks}\label{sec:unimod}

We first need to recall the general framework of \cite{aldouslyons}.
We still denote by $G=(V,E)$ a possibly infinite graph with vertex set
$V$ and undirected edge set $E$. We always assume that the degrees are
finite, i.e. the graph is locally finite. 
A network is a graph $G=(V,E)$ together with a complete separable metric space $\Xi$ called the mark space and maps from $V$ and $\orE$ to $\Xi$. Images in $\Xi$ are called marks. Each edge is given two marks, one associated to each of its orientation. 
A rooted network $(G,\circ)$ is a network with a distinguished vertex $\circ$ of $G$ called the root. A rooted isomorphism of rooted networks is an isomorphism of the underlying networks that takes the root of one to the root of the other. We do not distinguish between a rooted network and its isomorphism class denoted by $[G,\circ]$. Indeed, it is shown in \cite{aldouslyons} how to define a canonical representative of a rooted-isomorphism class.

Let $\Gcal_*$ denote the set of rooted isomorphism classes of rooted connected locally finite networks. Define a metric on $\Gcal_*$ by letting the distance between $[G_1,\circ_1]$ and $[G_2,\circ_2]$ be $1/(1+\alpha)$ where $\alpha$ is the supremum of those $r>0$ such that there is some rooted isomorphism of the balls of graph-distance radius $\lfloor r \rfloor$ around the roots of $G_i$ such that each pair of corresponding marks has distance less than $1/r$. $\Gcal_{*}$ is separable and complete in this metric \cite{aldouslyons}.

Similarly to the space $\Gcal_*$, we define the space $\Gcal_{**}$ of isomorphism classes of locally finite connected networks with an ordered pair of distinguished vertices and the natural topology thereon.

\begin{definition}\cite{aldouslyons}
Let $\rho$ be a probability measure on $\Gcal_*$. We call $\rho$ unimodular if it obeys the Mass-Transport Principle: for Borel $f:\Gcal_{**}\to [0,\infty]$, we have
\BEA
\label{eq:MTP}\int \sum_{x\in V} f(G,\circ, x) d\rho([G,\circ]) = \int \sum_{x\in V} f(G,x,\circ) d\rho([G,\circ]).
\EEA
Let $\Ucal$ denote the set of unimodular Borel probability measures on $\Gcal_*$.
\end{definition}

We now illustrate this concept with examples from our problem: a (possibly disconnected) finite graph $G=(V,E)$ with a vector of constraints $\bw=(w_v,\: v\in V)$ is a network with mark space $\RR$ and vertices have marks given by the vector $\bw$ (and oriented edges have mark zero). We still denote by $G$ the network and we write $G_x$ for the connected component of $x$ in $G$. For a (deterministic) finite network $G$, we define $U_G$ as the distribution on $\Gcal_*$ induced by the uniform measure on the vertices of $G$: more precisely, $U_G\left([G_x,x]\right)=\frac{1}{|V|}\sum_{y\in V}\ind\left( (G_y,y)\in [G_x,x]\right)$.
It is easy to verify that $U_G$ is unimodular for any finite network $G$ since the equality (\ref{eq:MTP}) is equivalent to an interchange of (finite) summation in:
\BEAS
\frac{1}{|V|} \sum_{\circ\in V}\sum_{x\in V} f(G,\circ,x).
\EEAS

If $\rho$ is a unimodular probability measure on $\Gcal_*$ with marks on vertices corresponding to their degree constraints, it will be useful to construct other unimodular measures from $\rho$. 
We will always start with $\rho$ a unimodular probability measure on $\Gcal_*$ with marks on vertices corresponding to their degree constraints (and marks on edges being zero).
We will take as mark spaces: $\Xi =\RR_+^{\N}$ so that the initial mark of vertex $v$ is simply $(w_v,0\dots)$.
Note that the map $\Rcal_G$ is still perfectly defined for any locally finite network $G$ and is isomorphism invariant, in the sense that for a graph isomorphism $\varphi$, we have $\Rcal_{\varphi(G)}(\bY)=\Rcal_G(\bY\circ \varphi^{-1})$ where vectors in $\bZ\in \RR^{\orE}$are viewed as maps from $\orE$ to $\RR$. Hence, we define $\Rcal_{[G,\circ]}=\Rcal_{G'}$ where $(G',o')$ is the canonical representative of $[G,\circ]$. To simplify notations, we write $\Rcal_G=\Rcal_{[G,\circ]}$.
We define $\bY^0(z)=\bold{0}$ and for $k\geq 0$, $\bY^{k+1}(z)=z\Rcal_G(\bY^k(z))$.
It is easy to show that if we set the mark for each edge $\ore$ to $(\bY^0_{\ore}(z), \bY^1_{\ore}(z),\dots, \bY^k_{\ore}(z),0,\dots)$ we should still get an unimodular probability on $\Gcal^*$. Indeed, write explicitly the dependence in the mark and define
$g(G,\circ, x, \bY(z))= f(G,\circ, x, z\Rcal_G(\bY(z)))$. If we proved that the measure $\rho$ with the marks $(\bY^0_{\ore}(z), \bY^1_{\ore}(z),\dots, \bY^{k-1}_{\ore}(z),0,\dots)$ is unimodular, we get
\BEAS
\int \sum_{x\in V} f(G,\circ, x, \bY^k(z)) d\rho([G,\circ]) &=&\int \sum_{x\in V} g(G,\circ, x, \bY^{k-1}(z)) d\rho([G,\circ])  \\
&=& \int \sum_{x\in V} g(G,x,\circ,\bY^{k-1}(z)) d\rho([G,\circ])\\
&=&\int \sum_{x\in V} f(G,x,\circ, \bY^k(z)) d\rho([G,\circ]),
\EEAS
and the claim for $k$ follows.

By Proposition \ref{prop:localmon}, we can define $\lim_{k \to \infty}\uparrow \bY^{2k}(z)=\bY^-(z)$ and $\lim_{k \to \infty}\downarrow \bY^{2k+1}(z) = \bY^+(z)$, with $\bY^+(z)=z\Rcal_G(\bY^-(z))$ and $\bY^-(z)=z\Rcal_G(\bY^+(z))$ $\rho$-a.s. 
We now define for any $\bY\in \RR_+$:
\BEAS
f(G,\circ,x,\bY) &=& \frac{Y_{x\to \circ} \Rcal_{x\to \circ}(\bY)}{1+Y_{x\to \circ} \Rcal_{x\to \circ}(\bY)}, \mbox{ for $\circ$ and $x$ adjacent,}
\EEAS
and $f(G,\circ,x,\bY)=0$ if $\circ$ and $x$ are not adjacent.
Note that we have: $0=\bY^{0}(z)\leq \bY^-(z)\leq\bY^+(z)\leq z=\bY^1(z) $, in particular all quantities are finite, hence by definition:
\BEAS
\int \Dcal_{\circ}(\bY^-(z))d\rho([G,\circ])&=&\int\sum_{x\in V}f(G,\circ,x,\bY^{-}(z))d\rho([G,\circ])\\ 
&=&\int\sum_{x\in \d \circ}\frac{Y^-_{x\to \circ}(z) \Rcal_{x\to \circ}(\bY^-(z))}{1+Y^-_{x\to \circ}(z) \Rcal_{x\to \circ}(\bY^-(z))}d\rho([G,\circ])\\
&=&\int\sum_{x\in \d \circ}\frac{\Rcal_{\circ\to x}(\bY^+(z)) Y^+_{\circ\to x}(z) }{1+\Rcal_{\circ\to x}(\bY^+(z)) Y^+_{\circ\to x}(z)}d\rho([G,\circ])\\
&=&\int \sum_{x\in V}f(G,x,\circ,\bY^+(z))d\rho([G,\circ])
\EEAS
where we used the identity: $Y^-_{x\to \circ}(z) \Rcal_{x\to
  \circ}(\bY^-(z))=\Rcal_{\circ\to x}(\bY^+(z)) Y^+_{\circ\to
  x}(z)$. Since $f$ is continuous in $\bY$ and bounded by one, we can
apply the dominated convergence theorem and the Mass-Transport
Principle with marks $\bY^{2k+1}(z)$ which converge to $\bY^+(z)$, to get:
\BEAS
\int \Dcal_{\circ}(\bY^-(z))d\rho([G,\circ])&=&\int \sum_{x\in V_G}f(G,\circ,x,\bY^+(z))d\rho([G,\circ])\\
&=&\int \Dcal_{\circ}(\bY^+(z))d\rho([G,\circ]),
\EEAS
Write $\odeg(\rho)$ 
for the expectation of the degree of the root with respect to $\rho$.
If $\odeg(\rho)<\infty$, then the above expectations are finite and since by Proposition \ref{prop:localmon}, we have $\rho$-a.s. $\Dcal_{\circ}(\bY^-)\leq \Dcal_{\circ}(\bY^+)$, it follows that we have indeed equality and then by the strict monotonicity of $\Dcal_\circ$, we get $\bY^-_{x\to \circ}(z)=\bY^+_{x\to \circ}(z)$, for all $x\in \d \circ$ and $z>0$, $\rho$-a.s. Then by Lemma 2.3 \cite{aldouslyons}, this directly implies that $\bY^-(z)=\bY^+(z)$ $\rho$-a.s. We denote it simply by $\bY(z)$ and we have $\bY(z)=z\Rcal_G(\bY(z))$.

Let now $\rho^Y$ be a probability measure on $\Gcal_*$ with marks on vertices $\bw$ and on edges $\bY(z)$ such that $\bY(z)=z\Rcal_G(\bY(z))$ $\rho^Y$-a.s. We assume that the marginal law $\rho$ of $\rho^Y$ where marks on edges are ignored (and only marks on vertices $\bw$ are kept) is unimodular.
Since $\bold{0}\leq \bY(z)\leq z$, we have by a simple induction $\bY^-(z)\leq \bY(z)\leq\bY^+(z)$ so that $\bY^-(z)=\bY^+(z)=\bY(z)$ $\rho^Y$-a.s. if $\odeg(\rho)<\infty$. Hence we proved the following proposition:

\begin{proposition}\label{prop:fpoint}
Let $\rho$ be a unimodular probability measure on $\Gcal_*$ with marks on vertices corresponding to degree constraints $\bw$. If $\odeg(\rho)<\infty$, then the fixed point equation $\Yb(z)=z\Rcal_G(\Yb(z))$ admits a unique solution $\Yb(z)$ for any $z>0$ for $\rho$-almost every $(G,\bw)$.
\end{proposition}

This proposition allows us to extend the fixed point equation of
Proposition \ref{prop:Recfinite} to an infinite setting. In the case
where a sequence of finite graphs $G_n$ admits a random weak limit
$\rho$ concentrated on trees (to be defined in Section
\ref{sec:fininf}), the solution $\Yb(z)$ will allow us to define the
limit of $\mu_{G_n}^z$.
We now concentrate on the $z\to \infty$ limit and derive a result
similar to Proposition \ref{prop:MGfinite} for unimodular networks.

We first introduce a convenient definition:
\begin{definition}
Let $\rho$ be a unimodular probability measure on $\Gcal_*$ with mark space $\Xi$. Let $\bM$ be maps from $\orE$ to $\Xi$. We say that $\bM$ is spatially invariant if the measure obtained from $\rho$ by adding the marks $\bM$ is still unimodular.
With a slight abuse of notation (i.e. we still denote by $\rho$ the measure with marks in $\Xi\times \Xi$), spatial invariance is characterised by: for Borel $f:\Gcal_{**}\times \Xi^{ \orE}\to [0,\infty]$, we have:
\BEA
\label{eq:MTPext}\int \sum_{x\in V_G} f(G,\circ, x,\bM) d\rho([G,\circ]) = \int \sum_{x\in V_G} f(G,x,\circ,\bM) d\rho([G,\circ]).
\EEA
\end{definition}
For example, we already proved that $\bY^k(z)$ is spatially invariant for our generic $\rho$.

Note that the maps $\Qcal_G$ and $\Pcal_G$ are well-defined for any locally finite network $G$ and we can define them for the isomorphism class $[G,\circ]$ thanks to its canonical representative, as we did above for $\Rcal_G$.
We can now define for any $\bI\in \{0,1\}^{\orE}$, 
\BEAS
F_{\circ}(\bI) = w_{\circ}\ind(\sum_{x\in \d \circ}\Pcal_{\circ \to x}(\bI)\geq w_{\circ}+1)+w_{\circ}\wedge \sum_{x\in \d \circ}I_{x\to \circ}.
\EEAS
Recall that for any $\bY(z)$, the 'degree' of $v$: $\Dcal_v(\bY(z))$ is defined by (\ref{eq:defD}).
We now prove the following result
\begin{proposition}\label{prop:MG}
Let $\rho$ be a unimodular probability measure on $\Gcal_*$ with $\odeg(\rho)<\infty$. Let $\bY(z)$ be the unique solution to $\bY(z)=z\Rcal_G(\bY(z))$ as defined in Proposition \ref{prop:fpoint}, then $\bY(z)\uparrow\bY$ as $z$ tends to infinity, where $\bY$ is the minimal solution to $\bY=\Qcal_G\circ \Rcal_G(\bY)$ and we have
\BEA\label{eq:inflim}
\lim_{z\to \infty}\int \Dcal_\circ(\Yb(z)) d\rho([G,\circ]) =\int \Dcal_\circ(\Yb) d\rho([G,\circ]) = \inf\left\{ \int F_\circ(\bI)d\rho([G,\circ]) \right\},
\EEA
where the infimum is over all spatially invariant solution of $\bI=\Pcal_G\circ\Pcal_G(\bI)$. 
Moreover, when $\rho$ is concentrated on trees, we can restrict the infimum over all spatially invariant solution of $\bI=\Pcal_G\circ\Pcal_G(\bI)$ which are locally independent, i.e. such that the components of the vector $\bI$ restricted to $\d \circ$ are independent.
\end{proposition}

In the case where $\rho$ is concentrated on finite trees, we recover exactly Proposition \ref{prop:MGfinite} since in this case there is a unique solution to $\bI=\Pcal_G\circ\Pcal_G(\bI)=\Pcal_G(\bI)$. One can check that in the example of Figure \ref{fig:dogx}, we have:
\BEAS
\int F_\circ(\bI_1)d\rho([G,\circ])=\frac{9}{8}, \int F_\circ(\bI_2)d\rho([G,\circ])=\frac{14}{8}, \mbox{ and }\int F_\circ(\bI_3)d\rho([G,\circ])=\frac{10}{8},
\EEAS
where the last quantity is the mean degree in a maximum spanning
subgraph. Note that the condition $\bI=\Pcal_G\circ \Pcal_G(\bI)$ is
crucial for (\ref{eq:inflim}) to be valid.

In the infinite setting, there might exist different spatially
invariant solutions to $\bI=\Pcal_G\circ \Pcal_G(\bI)$ even in the
case where $\rho$ is concentrated on trees.
In the particular case of branching processes, the computation of
$\int F_\circ(\bI)d\rho([G,\circ])$ for each of them is easy and we do
it in the following section. In order to conclude and get the
asymptotic size of a maximum spanning subgraph, we still need to show
that we can invert the limits in $z$ and $n$ and this is done in
Proposition \ref{prop:inverslim}.

\begin{remark}
Proposition \ref{prop:MG} is valid for any unimodular probability measure $\rho$, even if $\rho$ does not concentrate on trees. However, we can interpret $\Dcal_\circ(\Yb(z))$ as the degree of the root in a spanning subgraph taken at random according to the Gibbs measure with activity $z$, only when $\rho$ concentrates on trees. 
\end{remark}

\begin{proof}
For $x\in [0,\infty]$, we define $I(x)=\ind(x=\infty)$ and $\bI(.)$ acts similarly on vectors (componentwise).
For $\bI\in \{0,1\}^E$, we define $\infty_{\bI}$ as the vector in $\{0,\infty\}^E$ having the same positive components as $\bI$. Note that we have $\infty_{\bI(\bx)}\leq \bx$ for any $\bx\in [0,\infty]^E$.

\begin{lemma}
Let $\rho$ be a unimodular probability measure on $\Gcal_*$ with $\odeg(\rho)<\infty$. For $\bY$ spatially invariant, we define $\bY'=\Qcal_G\circ \Rcal_G(\bY)$. Then $\bI(\bY') = \Pcal_G\circ\Pcal_G(\bI(\bY))$ and
\begin{itemize}
\item if $\bY\leq \bY'$ then $\int \Dcal_\circ(\bY)d\rho\geq \int F_\circ(\bI(\bY))d\rho$;
\item if $\bY\geq \bY'$ then $\int \Dcal_\circ(\bY)d\rho\leq \int F_\circ(\bI(\bY))d\rho$.
\end{itemize}
\end{lemma}
\begin{proof}
We have
\BEAS
\int \Dcal_\circ(\bY)d\rho &=& \int\left(\sum_{x\in \d \circ}\frac{Y_{x\to\circ}\Rcal_{x\to\circ}(\bY)}{1+Y_{x\to\circ}\Rcal_{x\to\circ}(\bY)}\ind(Y_{x\to\circ}<\infty)+w_{\circ}\wedge \sum_{x\in \d \circ} I(Y_{x\to \circ}) \right)d\rho([G,\circ])\\
&\geq& \int\left(\sum_{x\in \d \circ}\frac{Y'_{x\to\circ}\Rcal_{x\to\circ}(\bY)}{1+Y'_{x\to\circ}\Rcal_{x\to\circ}(\bY)}\ind(Y'_{x\to\circ}<\infty)+w_{\circ}\wedge \sum_{x\in \d \circ} I(Y_{x\to \circ}) \right)d\rho([G,\circ]),
\EEAS
if $\bY'\leq \bY$.
Let $\bX=\Rcal_G(\bY)$ so that $\bY'=\Qcal_G(\bX)$ and we have thanks to the MTP applied to the first term:
\BEAS
\int \Dcal_\circ(\bY)d\rho &\geq& \int\left(\sum_{x\in \d \circ}\frac{X_{x\to\circ}\Qcal_{x\to\circ}(\bX)}{1+X_{x\to\circ}\Qcal_{x\to\circ}(\bX)}\ind(\Qcal_{x\to\circ}(\bX)<\infty)+w_{\circ}\wedge \sum_{x\in \d \circ} I(Y_{x\to \circ}) \right)d\rho([G,\circ])\\
&=& \int \left( w_{\circ} \ind\left( \sum_{x\in \d \circ}\ind(X_{x\to \circ}>0)\geq w_{\circ}+1\right)+w_{\circ}\wedge \sum_{x\in \d \circ} I(Y_{x\to \circ}) \right)d\rho([G,\circ]).
\EEAS
\end{proof}

Let $\bI$ be a spatially invariant solution of $\bI =\Pcal_G\circ\Pcal_G(\bI)$.
We define $\bW^0=\infty_{\bI}$ and for $k\geq 0$, $\bW^{k+1} = \Qcal_G\circ \Rcal_G(\bW^k)$. By previous lemma, we have
\BEAS
\bI(\bW^{k+1})=\Pcal_G\circ\Pcal_G(\bI(\bW^k))=\bI.
\EEAS
Hence we have $\bW^0\leq \bW^1$ and since both $\Qcal_G$ and $\Rcal_G$ are both non-increasing the sequence $\bW^k$ is non-decreasing and we denote by $\bW^\bI$ its limit. We clearly have $\bI(\bW^\bI)\geq \bI$ and $\bW^\bI=\Qcal_G\circ \Rcal_G(\bW^\bI)$ so that by previous lemma:
\BEAS
\int \Dcal_\circ(\bW^\bI)d\rho([G,\circ]) = \int F_\circ(\bI(\bW^\bI))d\rho([G,\circ]). 
\EEAS
Moreover since $\bW^{k+1}\geq \bW^k$, previous lemma implies
\BEAS
\int F_\circ(\bI)d\rho([G,\circ]) =\int F_\circ(\bI(\bW^k))d\rho([G,\circ]) \geq \int \Dcal_{\circ}(\bW^k)d\rho([G,\circ]).
\EEAS
Moreover, recall that $\Dcal_\circ$ is increasing so that by the monotone convergence theorem, we have
\BEAS
\lim_{k\to \infty}\int \Dcal_{\circ}(\bW^k)d\rho([G,\circ]) = \int \Dcal_{\circ}(\bW^\bI)d\rho([G,\circ]).
\EEAS
Hence we proved:
\begin{lemma}\label{lem:monotI}
Let $\bI$ be a spatially invariant solution of $\bI =\Pcal_G\circ\Pcal_G(\bI)$. Then $\bW^\bI=\Qcal_G\circ\Rcal_G(\bW^\bI)$ and is such that $\bI(\bW^\bI)\geq \bI$ and
\BEAS
\int \Dcal_\circ(\bW^\bI)d\rho([G,\circ]) = \int F_\circ(\bI(\bW^\bI))d\rho([G,\circ])\leq \int F_\circ(\bI)d\rho([G,\circ]).
\EEAS
\end{lemma}

\begin{lemma}\label{lem:WI}
If $\bY$ is a spatially invariant solution to $\bY=\Qcal_G\circ \Rcal_G(\bY)$, then we have $\bY=\bW^{\bI(\bY)}$.
\end{lemma}
\begin{proof}
Since $\bW^0\leq \bY$, we have $\bW^k\leq \bY$ and $\bW^{\bI(\bY)}\leq \bY$, in particular $\bI(\bW^{\bI(\bY)})\leq \bI(\bY)$. By previous lemma, we also have $\bI(\bW^{\bI(\bY)})\geq \bI(\bY)$ so that we have indeed equality and hence
\BEAS
\int \Dcal_\circ(\bW^{\bI(\bY)})d\rho([G,\circ]) = \int F_\circ(\bI(\bW^{\bI(\bY)}))d\rho([G,\circ]) = \int F_\circ(\bI(\bY))d\rho([G,\circ])=\int \Dcal_\circ(\bY)d\rho([G,\circ]).
\EEAS
Since $\Dcal_\circ$ is increasing and $\bW^{\bI(\bY)}\leq \bY$, we finally get equality.
\end{proof}

We are now ready to prove Proposition \ref{prop:MG}.
We first prove that $z\mapsto \frac{\bY(z)}{z}$ and $z\mapsto \bY(z)$ are respectively non-increasing and non-decreasing. We need only to prove that this result is correct for $\bY^k$ for any $k\geq 0$ and this is shown by a simple induction on $k$ as follows: consider $z\leq z'$; if $\bY^k(z)\leq \bY^k(z')$ then we have $\frac{\bY^{k+1}(z)}{z} \geq \frac{\bY^{k+1}(z')}{z'}$ since $\Rcal_G$ is non-increasing by Proposition \ref{prop:localmon}; if $\frac{\bY^{k}(z)}{z} \geq \frac{\bY^{k}(z')}{z'}$ then we have $\bY^{k+1}(z)\leq \bY^{k+1}(z')$ since $(z, \Xb)\mapsto z\Rcal_G(z\Xb)$ is increasing in $z$ and non-increasing in $\Xb$ still by Proposition \ref{prop:localmon}. 

Hence we can define $ \lim_{z\to \infty}\uparrow \bY(z)=\bY \in [0,\infty]^{\orE}$ and $ \lim_{z\to \infty}\downarrow \frac{\bY(z)}{z}=\bX \in [0,1]^{\orE}$ so that by passing to the limit in $\frac{\bY^{k+1}(z)}{z}=\Rcal_G(\bY^k(z))$ and $\bY^{k+1}(z) = z\Rcal_G\left(z \frac{\bY^{k}(z)}{z}\right)$, we have $\bX=\Rcal_G(\bY)$ and $\bY =\Qcal_G(\bX)$. 
Again since the pointwise limit of a sequence of measurable functions is measurable, we see that $(\bY,\bX)$ is spatially invariant and $\bY=\Qcal_G\circ \Rcal_G(\bY)$.

Let now $\rho^Z$ be a unimodular probability measure on $\Gcal_*$ with marks on vertices $\bw$ and on edges $\bZ$ such that $\bZ=\Qcal_G\circ \Rcal_G(\bZ)$ $\rho^Z$-a.s. The marginal $\rho$ of $\rho^Z$ where marks on edges are ignored is still unimodular and by definition $\rho$ and $\rho^Z$ are coupled. By definition, we have for any $\bX\in [0,1]^{\orE}$ and $z>0$, $z\Rcal_G(z\bX)\leq \Qcal_G(\bX)$ so that an easy induction on $k$ shows that: $\bY^{2k}(z)\leq \bZ$, $\rho^{\bZ}$-a.s. Letting first $k$ and then $z$ tend to infinity, we see that $\bY\leq \bZ$, $\rho^{\bZ}$-a.s.

Let $\bY^0=0$ and for $k\geq 0$, $\bY^{k+1}=\Qcal_G\circ \Rcal_G(\bY^k)$ so that $\bY^k$ is non-decreasing and (by previous result) converges to $\bY$, the smallest solution to $\bY= \Qcal_G\circ \Rcal_G(\bY)$. 
We have
\BEAS
\lim_{z\to \infty}\int \Dcal_\circ(\Yb(z)) d\rho([G,\circ]) =\int \Dcal_\circ(\bY)d\rho=\int F_\circ(\bI(\bY))d\rho([G,\circ]).
\EEAS

Let $\bI$ be a spatially invariant solution of $\bI=\Pcal_G\circ\Pcal_G(\bI)$, then by Lemma \ref{lem:monotI}, we have
\BEAS
\int F_\circ(\bI(\bW^\bI))d\rho([G,\circ])\leq \int F_\circ(\bI)d\rho([G,\circ]),\mbox{ and since $\bW^\bI\geq \bY$},
\EEAS
by Lemma \ref{lem:WI}, we get,
\BEAS
\int \Dcal_\circ(\bY)d\rho([G,\circ])\leq \int \Dcal_\circ(\bW^{\bI})d\rho([G,\circ])=\int F_\circ(\bI(\bW^\bI))d\rho([G,\circ]).
\EEAS

In the case where $\rho$ is concentrated on trees, it is easily seen that $\bY$ is locally independent so that the last claim of the proposition follows.

\end{proof}
\begin{remark}
With the notation introduced above, we define $\bI^k=\bI(\bY^k)$, a sequence that monotonically converges to $\bI^-$ the smallest solution to $\bI^-=\Pcal_G\circ\Pcal_G(\bI^-)$. We have $\bI(\bY)\geq \bI^-$. This implies that $\infty_{\bI^-}\leq \bY$ so that we have $\bY =\bW^{\bI^-}$. Hence by Lemma \ref{lem:monotI}, we get
\BEAS
\int \Dcal_\circ(\bY)d\rho([G,\circ]) =\int F_\circ(\bI(\bY))d\rho([G,\circ])\leq \int F_\circ(\bI^-)d\rho([G,\circ]).
\EEAS
In particular, it is NOT necessary the case that $\bI(\bY)$ achieving
the minimum in (\ref{eq:inflim}) is the smallest solution of $\bI=\Pcal_G\circ\Pcal_G(\bI)$.
\end{remark}

\section{From finite graphs to unimodular trees}\label{sec:fininf}

In this section, we show that the analysis made in previous section on infinite unimodular networks allow to compute the maximum size of a spanning subgraph of a sequence of finite graphs when the number of vertices tend to infinity.
We consider a sequence of finite networks $(G_n = (V_n,E_n,\bw^n))_{n\in \N}$.
For such sequence, we will use the notion of local weak convergence introduced by \cite{besc01} and \cite{alst04} (see also \cite{aldouslyons}).
Recall that we show in Section \ref{sec:unimod} that uniform rooting is a natural procedure to construct a unimodular measure on $\Gcal_*$ from a finite network $G$. This measure is denoted by $U_G$. For $\rho$ a unimodular probability measure on $\Gcal_*$, we say the random weak limit of $G_n$ is $\rho$ if $U_{G_n}$ converges weakly to $\rho$ (see \cite{billingsley} for details on weak convergence). 

\begin{proposition}\label{prop:inverslim}
Let $(G_n=(V_n,E_n,\bw^n))_{n\in \N}$ be a sequence of finite networks with random weak limit $\rho\in\BUT$, with $|E_n|=O(|V_n|)$ and $\sup_{v\in V_n}w^n_v\leq K$ for a fixed $K>1$. If $\rho$ is concentrated on trees, then we have
\BEAS
\lim_{n\to\infty}\frac{1}{|V_n|} M(G_n) = \int \Dcal_\circ(\bY)d\rho([G,\circ]).
\EEAS
\end{proposition}
\begin{proof}
We start be stating a simple consequence of spatial Markov property, see Lemma 4 in \cite{salez}.
\begin{lemma}\label{lem:treeapprox}
For a finite simple graph $G=(V,E)$ with a vector of degree constraints $\bw$, let $v\in V$. If the induced subgraph obtained by keeping only vertices at graph distance at most $2k+2$ from $v$ is a tree, then we have for any $z>0$,
\BEAS
\Dcal_v(\bY^{2k}(z))\leq D_v^z=\sum_{e\in \d v}\mu_G^z\left(B_e=1\right)\leq \Dcal_v(\bY^{2k+1}(z)).
\EEAS
\end{lemma}

For a finite graph $G=(V,E)$ and $v\in V$, we denote by $\chi_k(G,v)$ the indicator function that the induced subgraph obtained by keeping only vertices at graph distance at most $2k+2$ from the vertex $v$ is a tree.
By Lemma \ref{lem:treeapprox}, we have
\BEAS
\Dcal_{v}(\bY^{2k}(z))\chi_k(G,v)\leq D_{v}^z=\sum_{e\in \d v}\mu_G^z\left(B_e=1\right)\leq \Dcal_v(\bY^{2k+1}(z))\chi_k(G,v)+(1-\chi_k(G,v))w_v
\EEAS
Assuming that, we have $w_v\leq K$ for all $v$, we get
\BEAS
\frac{1}{|V|} \sum_{v\in V}\Dcal_{v}(\bY^{2k}(z))\chi_k(G,v)\leq \frac{1}{|V|} \sum_{v\in V}D_{v}^z\leq \frac{1}{|V|} \sum_{v\in V}\Dcal_{v}(\bY^{2k+1}(z))\chi_k(G,v)+\frac{K}{|V|}\sum_{v\in V}(1-\chi_k(G,v)).
\EEAS
Let now $G_n$ be a sequence of finite networks with random weak limit $\rho$.
For simplicity, let denote by $\rho_n$ be the probability measure $U_{G_n}$ on $\Gcal_*$ obtained by taking as root of $G_n$ a vertex in $V_n$ uniformly at random. We can rewrite previous inequalities as follows:
\BEA\label{eq:treap}
\int \Dcal_{\circ}(\bY^{2k}(z))\chi_k(G_n,\circ)d\rho_n\leq\int D_{\circ}^zd\rho_n\leq 
\int \Dcal_{\circ}(\bY^{2k+1}(z))\chi_k(G_n,\circ)d\rho_n+K\int\left( 1-\chi_k(G_n,\circ)\right) d\rho_n
\EEA
Note that we have $\Dcal_{\circ}(\bY^{k}(z))\leq K$ and by definition of the metric on $\Gcal_*$, the left-hand side and the right-hand side which depends only on the ball of (graph-distance) radius $2k+2$ around the root are continuous on $\Gcal_*$ so that we can take the limit in (\ref{eq:treap}). Since $\rho$ is concentrated on trees, the function $\chi_k$ is one on the support of $\rho$ and we get
\BEAS
\int \Dcal_{\circ}(\bY^{2k}(z))d\rho\leq\liminf_n\int D_{\circ}^zd\rho_n\leq\limsup_n\int D_{\circ}^zd\rho_n\leq 
\int \Dcal_{\circ}(\bY^{2k+1}(z))d\rho.
\EEAS
By Proposition \ref{prop:MG}, we can take the limit as $k$ tends to infinity, to finally have for any $z>0$,
\BEA\label{eq:limn}
\lim_n\int D_{\circ}^zd\rho_n= \int \Dcal_{\circ}(\bY(z))d\rho.
\EEA
Recall that by definition,
\BEAS
\int D_{\circ}^zd\rho_n=\frac{1}{|V_n|}\sum_{v\in V_n} D^z_{v}\leq \frac{1}{|V_n|}M(G_n),
\EEAS
so that we get by letting first $n$ tend to infinity, using (\ref{eq:limn}) and then letting $z$ tend to infinity
\BEAS
\int \Dcal_{\circ}(\bY)d\rho=\lim_{z\to\infty}\int \Dcal_{\circ}(\bY(z))d\rho\leq \liminf_n \frac{1}{|V_n|}M(G_n).
\EEAS
To prove the upper bound, we start with the following observation:
\BEAS
\frac{zP'_{G_n}(z)}{|V_n|P_{G_n}(z)} = \int D_\circ^z d\rho_n,
\EEAS
where the polynomial $P_{G_n}$ is the one appearing in the definition of Gibbs measure $\mu_{G_n}$, see (\ref{eq:gib}).
Differentiating once more, we get
\BEAS
z|V_n|\frac{d}{dz}\int D_\circ^z d\rho_n &=& \frac{zP'_{G_n}(z)}{P_{G_n}(z)}+\frac{z^2P''_{G_n}(z)}{P_{G_n}(z)}-\frac{z^2P'_{G_n}(z)^2}{P_{G_n}(z)^2}\\
&=& |V_n|\left( \int \left(D_\circ^z\right)^2 d\rho_n- \left(\int D_\circ^z d\rho_n\right)^2\right)\geq 0,
\EEAS
by Jensen's inequality.
In particular, we see that $z\mapsto \int D_\circ^z d\rho_n$ is non-decreasing in $z>0$ so that we have for any $1<s$,
\BEAS
\int_1^s z^{-1}\int D_\circ^z d\rho_n dz &\leq& \int_1^sz^{-1}dz \int D_\circ^s d\rho_n\\
\frac{1}{|V_n|} \log \frac{P_{G_n}(s)}{P_{G_n}(1)}&\leq& \log s \int D_\circ^s d\rho_n.
\EEAS
Moreover, we clearly have $P_{G_n}(1)\leq 2^{|E_n|}$ and $P_{G_n}(s)\geq s^{M(G_n)}$, so that we get
\BEAS
\frac{1}{|V_n|\log s}\left( M(G_n)\log s - |E_n| \log 2\right)\leq\int D_\circ^s d\rho_n,
\EEAS
or reorganising the terms
\BEAS
\frac{1}{|V_n|} M(G_n) \leq \int D_\circ^s d\rho_n + \frac{|E_n|}{|V_n|}\frac{\log 2}{\log s}
\EEAS
Hence by letting first $n$ tend to infinity (using the fact that $|E_n|=O(|V_n|)$ and (\ref{eq:limn})) and then letting $z$ tend to infinity, we get
\BEAS
\limsup_n\frac{1}{|V_n|} M(G_n) \leq \int \Dcal_\circ(\bY) d\rho,
\EEAS
which concludes the proof.
\end{proof}

\section{Computation for branching processes and asymptotics for large graphs}\label{sec:comp}

In this section, we consider bipartite graphs $G=(V, E)$ where $V$ can be divided into two disjoint sets $A$ and $B$ such that every edge connects a vertex in $A$ to one in $B$. We also denote by $\bw^A$ and $\bw^B$ the vectors of degree constraints associated to vertices in $A$ and $B$ respectively.
We are mainly interested in sequences of finite networks $(G_n=(V_n,E_n),\bw^A_n,\bw^B_n)_{n\in\N}$ simply denoted by $G_n$ in the sequel, whose size $|V_n|$ diverges to infinity and such that $|E_n|=O(|V_n|)$.

Our main assumption will be that the sequence of finite networks $(G_n)_{n\in \N}$ admits a random weak limit. Motivated by applications, we now describe a natural limit (arising when the network $G_n$ itself is random as described in the Introduction) which we call Bipartite Unimodular Galton Watson Trees (BUGWT).
A BUGWT is a branching process parametrised by two distributions: a distribution $\pi^A$ on $\N^2$ corresponding to joint distribution of the degree of vertices in $A$ and their associated constraints and a distribution $\pi^{B}$ on $\N^2$ corresponding to the joint distribution of the degree and mark of vertices in $B$. 
We will always assume that
\BEAS
m^A = \sum_{k,j\geq 1} k\pi^A_{k,j}<\infty &\mbox{ and }& m^B = \sum_{k,j\geq 1} k\pi^{B}_{k,j}<\infty.
\EEAS
In particular, the size-biased distributions associated to $\pi^A$ and $\pi^B$ are well defined:
\BEAS
\hpi^A_{n,\ell} =\frac{(n+1)\pi^A_{n+1,\ell}}{m^A} \mbox{ and, } \hpi^{B}_{n,\ell} =\frac{(n+1)\pi^{B}_{n+1,\ell}}{m^B}.
\EEAS
A BUGWT is then a branching process defined as follows: with probability $\frac{m^B}{m^A+m^B}$, the root is of type $A$ and has a random number of children and a random constraint drawn according to the distribution $\pi^A$; all odd generation genitors have a distribution for the couple offspring, constraint given by $\hpi^B$ and all even generation genitors this distribution is $\hpi^A$; similarly with probability $\frac{m^A}{m^A+m^B}$, the root is of type $B$ and has a random number of children and a random constraint drawn according to the distribution $\pi^B$; all odd generation genitors have a distribution for the couple offspring, constraint given by $\hpi^A$ and all even generation genitors this distribution is $\hpi^B$. The set of distributions of BUGWT is clearly a subset of the unimodular probability measure on $\Gcal_*$ and we denote it by $\BUT\subset \Ucal$.

Before stating our main result in this framework, we introduce some notations: we denote by $(D^A,W^A)$ (resp. $(N^A,W^A)$) a random variable with distribution $\pi^A$ (resp. $\hpi^A$) and by $(D^B,W^B)$ (resp. $(N^B,W^B)$) a random variable with distribution $\pi^{B}$ (resp. $\hpi^{B}$).
For $0\leq x\leq 1$, let $D^A(x)$ be the thinning of $D^A$ obtained by taking $D^A$ points and randomly and independently keeping each of them with probability $x$ so that
$\PP(D^A(x)=k) =\sum_{n\geq k}\sum_\ell\pi^A_{n,\ell}{n \choose k} x^k(1-x)^{n-k}$.
We define similarly $N^B(x), D^B(x), N^B(x)$.

\begin{theorem}\label{the:limM}
Let $(G_n)_{n\in \N}$ be a sequence of finite (simple) bipartite
networks with random weak limit $\rho\in\BUT$ such that $D^A$
(resp. $D^B$) and $W^A$ (resp. $W^B$) are independent with finite means.
If $|E_n|=O(|A_n|)$ where $A_n$ is one subset of the partition of vertices of $G_n$, then we have
\BEAS
\lim_{n\to\infty}\frac{1}{|A_n|} M(G_n) &=& 
\inf_{q\in [0,1]}\{\Fcal^A(q), q=g^A\circ g^B(q)\}=\inf _{q\in[0,1]} \Fcal^A(q),
\EEAS
where 
\BEAS
\Fcal^A(q) &=& \EE\left[W^A\wedge D^A(g^B(q))\right]+\frac{\EE[D^A]}{\EE[D^B]}\EE\left[ W^B\ind\left( D^B(q)\geq W^B+1\right)\right],\mbox{ with}\\
g^A(p) &=& \PP\left(N^A(p)<W^A\right) \mbox{ and, } g^B(q) = \PP\left(N^B(q)<W^B\right).
\EEAS
\end{theorem}
\begin{proof}
Let $\rho^A$ (resp. $\rho^B$) be the probability measure on $\Gcal_*$ obtained from $\rho$ by conditioning on the event that the root is of type $A$ (resp. type $B$): $\circ \in A$ (resp. $\circ \in B$). The fact that
\BEAS
\lim_{n\to\infty}\frac{1}{|A_n|} M(G_n) = \lim_{z\to \infty}\int \Dcal_\circ(\bY(z))d\rho^A([G,\circ]),
\EEAS
follows from Proposition \ref{prop:inverslim}. We will now use Proposition \ref{prop:MG} to compute the right-hand term.

We first adapt the expression in (\ref{eq:inflim}) to our bipartite case and express it as a function of $\rho^A$. 
For $x\in [0,\infty]$, we define $I(x)=\ind(x=\infty)$ and $\bI(.)$ acts similarly on vectors (componentwise).
By Lemma \ref{lem:Dcal}, we have for any spatially invariant $\bY$: 
\BEAS
\lefteqn{\int \Dcal_\circ(\bY)\ind(\circ \in A)d\rho([G,\circ])=}\\
&&\int\left(\sum_{x\in \d \circ}\frac{Y_{x\to\circ}\Rcal_{x\to\circ}(\bY)}{1+Y_{x\to\circ}\Rcal_{x\to\circ}(\bY)}\ind(Y_{x\to\circ}<\infty)+w_{\circ}\wedge \sum_{x\in \d \circ} I(Y_{x\to \circ}) \right)\ind(\circ \in A)d\rho([G,\circ])
\EEAS
We now consider a spatially invariant $\bY$ solution of $\bY=\Qcal_G\circ \Rcal_G(\bY)$ and let $\bX=\Rcal_G(\bY)$. Hence applying first the MTP and then using Lemma \ref{lem:Q}, we get
\BEAS
\lefteqn{\int\sum_{x\in \d
    \circ}\frac{Y_{x\to\circ}\Rcal_{x\to\circ}(\bY)}{1+Y_{x\to\circ}\Rcal_{x\to\circ}(\bY)}\ind(Y_{x\to\circ}<\infty)\ind(\circ
  \in A)d\rho([G,\circ])}\\
&& =\int w_{\circ} \ind\left( \sum_{x\in \d \circ, x\in A}\ind(X_{x\to \circ}>0)\geq w_{\circ}+1\right)d\rho([G,\circ])\\
&&= \int w_{\circ}\ind(\sum_{x\in \d \circ}\Pcal_{\circ \to x}(\bI(\bY))\geq w_{\circ}+1)\ind(\circ\in B)d\rho([G,\circ])
\EEAS
where we used the fact $\ind(X_{x\to \circ}>0)=\Pcal_{\circ \to x}(\bI(\bY))$ by (\ref{eq:Rcal0}).
By Proposition \ref{prop:MG}, we know that we should consider the minimal such $\bY$ so that dividing by $\int \ind(\circ \in A)d\rho([G,\circ])=\frac{\EE[D^B]}{\EE[D^A+D^B]}$, we obtain:
\BEA
\nonumber \lefteqn{\lim_{z\to \infty}\int \Dcal_\circ(\Yb(z)) d\rho^A([G,\circ])= }\\
\label{eq:inflimb} && \inf\left\{ \int w_{\circ}\wedge \sum_{x\in \d \circ}I_{x\to \circ}d\rho^A([G,\circ]) + \frac{\EE[D^A]}{\EE[D^B]}\int w_{\circ}\ind(\sum_{x\in \d \circ}\Pcal_{\circ \to x}(\bI)\geq w_{\circ}+1)d\rho^B([G,\circ]) \right\},
\EEA
where the infimum is over all spatially invariant solution of $\bI=\Pcal_G\circ\Pcal_G(\bI)$ which are locally independent.
Thanks to the Markovian nature of the BUGWT, these recursions simplify into a recursive distributional equation that we now describe. Indeed the law of $I_{x\to y}$ is a Bernoulli distribution with parameter say $q$ if $x\in A$ and $p$ otherwise. Then the determination of $p$ and $q$ is done as follows: let $I^{A\to B},I^{A\to B}_i$ be independent Bernoulli random variables with parameter $p$ and $I^{B\to A},I^{B\to A}_i$ be independent Bernoulli random variables with parameter $q$; let $(N^A, W^A)$ and $(N^B,W^B)$ be independent random variables distributed according to $\hpi^A$ and $\hpi^B$ respectively. Then we must have the following equality in laws:
\BEAS
I^{A\to B} \stackrel{d}{=} \ind\left( \sum_{i=1}^{N^A} I^{B\to A}_i<W^A\right)&\mbox{ and, }&
I^{B\to A} \stackrel{d}{=} \ind\left( \sum_{i=1}^{N^B} I^{A\to B}_i<W^B\right).
\EEAS
Taking the expectation, we get:
\BEAS
q&=&\EE\left[I^{A\to B} \right]=\PP\left(N^A(p)<W^A\right)=g^A(p)\\
p&=&\EE\left[I^{B\to A} \right]=\PP\left(N^B(q)<W^B\right)=g^B(q).
\EEAS
We see that $\Fcal^A(q)$ correspond exactly to the integral in (\ref{eq:inflimb}) for a solution of $\bI=\Pcal_G\circ\Pcal_G(\bI)$ and the first claim of the Proposition follows.

We define
\BEAS
\lefteqn{F(p,q) =}\\
&& \EE\left[W^A\right]-\EE\left[ \sum_{0\leq j\leq W^A-1}\frac{(W^A-j)p^j}{j!}\phi^{(j)}_A(1-p)\right]+\frac{\EE[D^A]}{\EE[D^B]}\left(\EE\left[W^B\right] -\EE\left[ W^B\sum_{0\leq j\leq W^B} \frac{q^j}{j!}\phi_B^{(j)}(1-q)\right]\right),
\EEAS
so that $\Fcal^A(q)=F(g^B(q),q)$.

Let $\phi_A$ (resp. $\phi_B$) be the generating functions of $D^A$ (resp. $D^B$).
After simple computations, we get (by convention, the sum over the empty set is zero):
\BEAS
g^A(x) &=& \EE\left[ \sum_{0\leq j\leq W^A-1}\frac{x^j}{j!}\frac{\phi_A^{(j+1)}(1-x)}{\phi_A^{(1)}(1)}\right]\:\mbox{ and }\\
h^A(x)&=&-\frac{dg^A}{dx}(x) = \EE\left[ \frac{x^{W^A-1}}{(W^A-1)!}\frac{\phi_A^{(W^A+1)}(1-x)}{\phi_A^{(1)}(1)}\ind(W^A\geq 1)\right],
\EEAS
and similarly for $g^B$ and $h^B$. 

Simple computations show that:
\BEAS
\frac{\d F}{\d p}(p,q) &=&\EE\left[ \sum_{1\leq j\leq W^A}\frac{p^{j-1}}{(j-1)!}\phi^{(j)}_A(1-p)\right]= \phi^{(1)}_A(1)g^A(p)\\
\frac{\d F}{\d q}(p,q) &=& \frac{\EE[D^A]}{\EE[D^B]}\EE\left[ \frac{q^{W^B}}{(W^B-1)!}\phi_B^{(W^B+1)}(1-q)\ind(W^B\geq 1)\right] \\
&=& \phi^{(1)}_A(1)qh^B(q).
\EEAS
Hence we have
\BEAS
\frac{d\Fcal^A}{dq}(q) &=& \phi^{(1)}_A(1)h^B(q)\left( q-g^A\circ g^B(q)\right)
\EEAS
Note that if $\PP\left(0<W^B<D^B\right)=0$ then $h^B(x)=0$ for all $x\in [0,1]$ and if $\PP\left(0<W^B<D^B\right)>0$, then $h^B(x)>0$ for all $0<x<1$.
Moreover, since $g^A\circ g^B(0)\geq 0$ and $1\geq g^A\circ g^B(1)$, we see that (in the case where $\Fcal^A$ is not constant) the local minimums of $\Fcal^A$ must satisfy $q=g^A\circ g^B(q)$ so that
\BEAS
\inf_{q\in [0,1]}\{\Fcal^A(q), q=g^A\circ g^B(q)\}=\inf _{q\in[0,1]} \Fcal^A(q).
\EEAS
\end{proof}

\section{Application to orientability of random hypergraphs}\label{sec:orient}

We are now in position to prove Theorem \ref{the:hyperg} claimed in the Introduction. 
The fact that the random hypergraph $H_{n,\lfloor cn \rfloor,h}$ converges weakly almost surely is standard \cite{kim08}.
Hence we can apply our Theorem \ref{the:limM} with
\BEAS
D^A=h,\quad W^A=\ell &\mbox{ and, }& D^B=\Poi(ch),\quad W^B =k,
\EEAS
where $h>\ell\geq 1$ and $k\geq 1$.
Note that $Q(x,y) = e^{-x}\sum_{j\geq y}\frac{x^j}{j!} = \PP(\Poi(x)\geq y)$ so that we have:
\BEAS
\Fcal^A(q,c) &=& \EE\left[ \ell\wedge \Bin(h,g^B(q))\right]+\frac{kQ(chq,k+1)}{c},\\
g^B(q) &=& 1-Q(chq,k),
\EEAS
where we denoted explicitly the dependence in $c$ in the expression of $\Fcal^A$ given in Theorem \ref{the:limM}.

With the notations taken in the proof of Theorem \ref{the:limM}, we
have for $h-\ell\geq 2$,
\BEAS
g^A(x) &=& \PP\left( \Bin(h-1,x)<\ell\right)\\
h^A(x) &=& \frac{(h-1)\dots (h-\ell)}{(\ell-1)!}x^{\ell-1}(1-x)^{h-\ell-1}\\
h^{A'}(x)=\frac{dh^A}{dx}(x)&=& \frac{(h-1)\dots (h-\ell)}{(\ell-1)!}x^{\ell-2}(1-x)^{h-\ell-2}\left(\ell-1+(h-2\ell)x\right),
\EEAS
and if $\ell=h-1$, $h^A(x)=(h-1)x^{h-2}$ and $h^{A'}(x) = (h-1)(h-2)x^{h-3}$.
Similarly, we have for $k\geq 2$ (and the convention $0!=1$),
\BEAS
g^B(x)&=& 1-Q(chx,k)\\
h^B(x)&=& ch e^{-chx}\frac{(chx)^{k-1}}{(k-1)!}\\
h^{B'}(x)=\frac{dh^B}{dx}(x)&=& (ch)^2 e^{-chx}\frac{(chx)^{k-2}}{(k-2)!}\left(1-\frac{chx}{k-1}\right).
\EEAS
If $k=1$, we have $h^B(x)=che^{-chx}$ and $h^{B'}(x)=-(ch)^2e^{-chx}$.

We also define 
\BEAS
\Delta(x) = x-g^A\circ g^B(x),
\EEAS
so that we get
\BEAS
\Delta'(x) &=& 1-h^B(x)h^A\circ g^B(x)\\
\Delta''(x) &=& h^A\circ g^B(x)\left(h^B(x)\right)^2\left( \frac{h^{A'}\circ g^B(x)}{h^{A}\circ g^B(x)}-\frac{h^{B'}(x)}{h^B(x)^2}\right)=h^A\circ g^B(x)\left(h^B(x)\right)^2\delta(x)
\EEAS
A simple calculation shows that both $\frac{h^{A'}}{h^A}$ and
$\frac{h^{B'}}{(h^B)^2}$ are non-increasing. Hence since $g^B$ is also
non-increasing, we see that $\delta$ is non-decreasing and $\Delta''$
can vanish at most once on $(0,1)$, so that $\Delta'$ admits at most
two zeros on $[0,1]$ and $\Delta$ at most three. Moreover since
$\Delta(0)=0$ and $\Delta(1)>0$, there are only two cases: either
$\Delta(x)\geq 0$ for all $x\in [0,1]$ or there exists $0\leq u<v<1$
such that $\Delta(x)\geq 0$ for $x\in [0,u]\cup[v,1]$ and
$\Delta(x)\leq 0$ for $x\in[u,v]$. In particular in this last case,
$u$ (resp. $v$) is a local maximum (resp. minimum) of $\Fcal^A$.

Let $x^*(c)$ be the largest solution of $\Delta(x)=0$, then we have
\BEAS
\min_{q\in [0,1]}\Fcal^A(q,c)=\min\{ \ell, \Fcal^A(x^*(c),c)\},
\EEAS
since $\Fcal^A(0,c)=\ell$ and in the first case $x^*(c)=0$ while in the second case $x^*(c)=v$.
A simple computation shows that as soon as $x^*(c)>0$, we have $\frac{d\Fcal^A}{dc}(x^*(c),c)<0$.
Hence we need to find $c$ such that $x^*(c)>0$ and $\Fcal^A(x^*(c),c)=\ell$. Denoting $\xi^*=chx^*$, we can rewrite this last equation as
\BEA\label{eq:cxi}
c\EE\left[\left( \ell-\Bin(h,1-Q(\xi^*,k))\right)^+ \right]=kQ(\xi^*,k+1).
\EEA
By definition, we have $\xi^* =ch\PP\left( \Bin(h-1,1-Q(\xi^*,k))<\ell\right)$, so that since $\xi^*>0$, we can eliminate the variable $c$ in (\ref{eq:cxi}) and get
\BEAS
hk =\xi^*\frac{\EE\left[\left( \ell-\Bin(h,1-Q(\xi^*,k))\right)^+ \right]}{Q(\xi^*,k+1)\PP\left( \Bin(h-1,1-Q(\xi^*,k))<\ell\right)},
\EEAS
which is exactly the equation appearing in Theorem \ref{the:hyperg}.
We now show that this equation has a unique positive solution. By contradiction, assume that $0<\xi_1<\xi_2$ are two solutions. Then applying the map $\xi\mapsto\PP\left( \Bin(h-1,1-Q(\xi^*,k))<\ell\right)$ we can define $x_1<x_2$ and then $c_i=\frac{\xi_i}{hx_i}$, so that we have $\Fcal^A(x_1,c_1)=\Fcal^A(x_2,c_2)=\ell$ with $x_i=x^*(c_i)$. Since $\Delta$ is a non-increasing function of $c$, it implies that $x^*(c)$ is a non-decreasing function of $c$ and $c_1<c_2$. But then $\Fcal^A(x^*(c),c)=\ell$ for all $c\in [c_1,c_2]$, contradicting the relation $\frac{d\Fcal^A}{dc}(x^*(c),c)<0$ valid for $x^*(c)>0$.

As noted in the Introduction, $H_{n,\lfloor cn \rfloor,h}$ is $(\ell,k)$-orientable if and only if $M(G_n) = \ell|A_n|$. In particular, for $c>c^*_{h,\ell,k}$ as defined in Theorem \ref{the:hyperg}, we see that Theorem \ref{the:limM} implies the result since $\lim_{n\to \infty}\frac{M(G_n)}{|A_n|}<\ell$.
However in the case $c<c^*_{h,\ell,k}$, there is still a difficulty
since Theorem \ref{the:limM} only shows that $\lim_{n\to
  \infty}\frac{M(G_n)}{|A_n|}=\ell$. In words, there might exist
$o(n)$ vertices in $A_n$ with degree strictly less than $\ell$ in $M(G_n)$. 
In order to prove that with high probability, this will not occur we need to rely on specific properties of the $h$-uniform hypergraph $H_{n,m,h}$.

This part of the proof is also done in \cite{gawoarxiv10} and we now
describe the main steps of the end of the proof. First, it will be
more convenient to work with a different model of random hypergraph
that we denote $H_{n,p,h}$: each possible $(h)$-hyperedge is included
independently with probability $p$. More precisely, given $n$ vertices
with $n\geq h$, we obtain $H_{n,p,h}$ by including each $h$-tuple of
vertices with probability $p$, independently for each of the
${n\choose h}$ tuples. Standard arguments show that if $p{n \choose
  h}=c n$, i.e. $p=ch/{n-1\choose h-1}$ then the $H_{n,p,h}$ model is
essentially equivalent to the $H_{n,\lfloor c n \rfloor,h}$ model \cite{kim08}.

Let $c<\tilde{c}<c^*_{h,\ell,k}$ and consider the bipartite graph $\tilde{G}_n$ obtained from $H_{n, \tilde{p},h}$ with $\tilde{p}=\tilde{c} h/{n-1\choose h-1}$.
Given a maximum admissible spanning subgraph $S(\tilde{G}_n)$, we say that a vertex
in $A_n$ (resp. in $B_n$) with degree $\ell$ (resp. $k$) in such a spanning subgraph is
saturated. An alternating path is a path in which the edges alternatively are covered in $S$ and uncovered in $S$.
If there exists a vertex $v$ in $A_n$ which is not saturated, we consider the
subgraph obtained as follows: remove the edges in $S$ adjacent to
$v$ and keep all alternating paths starting from $v$. We denote by $K$
the resulting bipartite graph.
It is easy to see that
\begin{itemize}
\item the degree in $K$ of $v$ is at least $h-\ell+1$;
\item if $v\in B_n\cap K$, then $v$ is saturated in $S$ and its degree
  in $K$ is $k+1$;
\item if $w\in A_n\cap K$ then if $w$ is saturated in $S$, its degree in
  $K$ is at least $h-\ell+1$, otherwise its degree is at least $h-\ell+2$.
\end{itemize}

We now use a density argument in order to show
that the size of $K$ is linear in $n$ w.h.p. 
Suppose that we are in a case where $k\geq 2$ (the case $h-\ell\geq 2$ is similar). 
Let $i$ be the number of vertices in $B_n\cap K$. Each such vertex is adjacent to $k$ covered edges in $S$ and 1 uncovered edge in $S$. Hence, we have $|A_n\cap K|\geq \frac{ki}{\ell}$. Hence by Lemma 4.1 of \cite{gawo10}, there exist $\gamma=\gamma(h, \frac{k}{\ell})$ such that $|K\cap B_n|\geq \gamma n$ (note that the probability space considered in Lemma 4.1 of \cite{gawoarxiv10} is slightly different from ours but asymptotically equivalent as shown by Lemma 3.1 of \cite{gawoarxiv10}).

For $c<c^*_{h,\ell,k}$, we now construct an admissible spanning subgraph on $H_{n,p,h}$ with $p=c h/{n-1\choose h-1}$ and saturating all vertices in $A_n$. Note that since $p<\tp$, there is a natural coupling between $H_{n,p,h}$ and $H_{n,\tp,h}$: we can obtain $H_{n,p,h}$ from  $H_{n,\tp,h}$ by removing each hyperedge independently with probability $\tp-p>0$.
Start with a maximum admissible spanning subgraph $M(\tilde{G}_n)$ as above. If this maximum spanning subgraph has all vertices in $\tilde{A}_n$ which are saturated, then we are done, since to obtain an admissible spanning subgraph on $H_{n,p,h}$ we need only to remove the hyperedges in $H_{n,\tp,h}$ which are not in $H_{n,p,h}$.
If this maximum spanning subgraph has $o(n)$ vertices in $\tilde{A}_n$ which are not saturated. We denote by $gap_n= \ell|\tilde{A}_n|-M(\tilde{G}_n)=o(n)$. 
Note that for any non saturated hyperedge, we can construct a subgraph $K$ as above. If we remove an hyperedge in $K$, then either this edge was non-saturated and $gap_n$ decreased by more than one or, this edge was saturated but then at least one nonsaturaded hyperedge increases its degree in the spanning subgraph. So that in all cases, $gap_n$ decreased by at least one. We then proceed as follows: attach to each hyperedge a uniform $[0,1]$ random variable independently from each other and denoted $U_a$, for $a\in A_n$. 
To obtain $H_{n,p,h}$, we remove all hyperedges $a$ such that $U_a\leq \tilde{p}-p$.
We can do this sequentially starting with the one with the lowest $U_a$.
Suppose that at the $k$-th step (with $k\geq 1$), there are still non saturated hyperedges, we consider the union of the subgraphs $K$ described above. We showed that their size is at least $\gamma n$. Hence with positive probability, the hyperedge removed will decrease the value of $gap_n$. If it reaches zero, then we are done, otherwise, we consider the new sets $K$ and proceed to the next step.
Suppose that at the end of the procedure, $gap_n$ is still positive. We see that we decreased $gap_n$ in total by a binomial random variable with a number of trials linear in $n$ and with a positive probability so that with high probability, this binomial random variable is larger than say $\epsilon n$ for a sufficiently small $\epsilon>0$. Hence with high probability $gap_n$ reached zero by the end of the procedure and this concludes the proof.

\section{Acknowledgements}
The author would like to thank Laurent Massouli\'e for pointing out a mistake in a preliminary version of this work.
The author acknowledges the support of the French Agence Nationale de la Recherche (ANR) under reference ANR-11-JS02-005-01 (GAP project).


\end{document}